\documentclass[sn-mathphys-num]{sn-jnl}

\usepackage{graphicx}%
\usepackage{multirow}%
\usepackage{amsmath,amssymb,amsfonts,bm,mathtools}%
\usepackage{amsthm}%
\usepackage{mathrsfs}%
\usepackage[title]{appendix}%
\usepackage{xcolor}%
\usepackage{textcomp}%
\usepackage{manyfoot}%
\usepackage{booktabs}%
\usepackage{algorithm}%
\usepackage{algorithmicx}%
\usepackage{algpseudocode}%
\usepackage{listings}%
\usepackage{planarforest}
\newcommand{\forestA}{
\tikz[planar forest] {

\node [b] at (0.0, 0.0) {  } 
;
}}

\newcommand{\forestB}{
\tikz[planar forest] {

\node [b] at (0.0, 0.0) {  } 
child {node [b] at (0.0, 1.0) {  }  
child {node [b] at (-0.5, 1.0) {  }  
}
child {node [b] at (0.5, 1.0) {  }  
}
}
;
}}

\newcommand{\forestC}{
\tikz[planar forest] {

\node [b] at (0.0, 0.0) {  } 
child {node [b] at (-0.5, 1.0) {  }  
}
child {node [b] at (0.5, 1.0) {  }  
child {node [b] at (0.0, 1.0) {  }  
}
}
;
}}

\newcommand{\forestD}{
\tikz[planar forest] {

\node [b] at (0.0, 0.0) {  } 
child {node [b] at (-1.0, 1.0) {  }  
child {node [b] at (0.0, 1.0) {  }  
}
}
child {node [b] at (0.0, 1.0) {  }  
}
child {node [b] at (1.0, 1.0) {  }  
child {node [b] at (-0.5, 1.0) {  }  
}
child {node [b] at (0.5, 1.0) {  }  
}
}
;
}}

\newcommand{\forestE}{
\tikz[planar forest] {

\node [b] at (0.0, 0.0) {  } 
;
}}

\newcommand{\forestF}{
\tikz[planar forest] {

\node [b] at (0.0, 0.0) {  } 
child {node [b] at (-0.5, 1.0) {  }  
}
child {node [b] at (0.5, 1.0) {  }  
}
;
}}

\newcommand{\forestG}{
\tikz[planar forest] {

\node [b] at (0.0, 0.0) {  } 
child {node [b] at (0.0, 1.0) {  }  
child {node [b] at (-0.5, 1.0) {  }  
}
child {node [b] at (0.5, 1.0) {  }  
}
}
;
}}

\newcommand{\forestH}{
\tikz[planar forest] {

\node [b] at (0.0, 0.0) {  } 
;
\node [b] at (1.0, 0.0) {  } 
child {node [b] at (0.0, 1.0) {  }  
}
;
}}

\newcommand{\forestI}{
\tikz[planar forest] {

\node [b] at (0.0, 0.0) {  } 
child {node [b] at (-0.5, 1.0) {  }  
}
child {node [b] at (0.5, 1.0) {  }  
child {node [b] at (0.0, 1.0) {  }  
}
}
;
}}

\newcommand{\forestJ}{
\tikz[planar forest] {

\node [b] at (0.0, 0.0) {  } 
child {node [b] at (0.0, 1.0) {  }  
}
;
\node [b] at (1.0, 0.0) {  } 
;
\node [b] at (2.0, 0.0) {  } 
child {node [b] at (-0.5, 1.0) {  }  
}
child {node [b] at (0.5, 1.0) {  }  
}
;
}}

\newcommand{\forestK}{
\tikz[planar forest] {

\node [b] at (0.0, 0.0) {  } 
child {node [b] at (-1.0, 1.0) {  }  
child {node [b] at (0.0, 1.0) {  }  
}
}
child {node [b] at (0.0, 1.0) {  }  
}
child {node [b] at (1.0, 1.0) {  }  
child {node [b] at (-0.5, 1.0) {  }  
}
child {node [b] at (0.5, 1.0) {  }  
}
}
;
}}

\newcommand{\forestL}{
\tikz[planar forest] {

\node [b] at (0.0, 0.0) {  } 
;
}}

\newcommand{\forestM}{
\tikz[planar forest] {

\node [b] at (0.0, 0.0) {  } 
;
}}

\newcommand{\forestN}{
\tikz[planar forest] {

\node [b] at (0.0, 0.0) {  } 
;
}}

\newcommand{\forestO}{
\tikz[planar forest] {

\node [b] at (0.0, 0.0) {  } 
child {node [b] at (0.0, 1.0) {  }  
}
;
}}

\newcommand{\forestP}{
\tikz[planar forest] {

\node [b] at (0.0, 0.0) {  } 
child {node [b] at (0.0, 1.0) {  }  
child {node [b] at (0.0, 1.0) {  }  
}
}
;
}}

\newcommand{\forestQ}{
\tikz[planar forest] {

\node [b] at (0.0, 0.0) {  } 
child {node [b] at (-0.5, 1.0) {  }  
}
child {node [b] at (0.5, 1.0) {  }  
}
;
}}

\newcommand{\forestR}{
\tikz[planar forest] {

\node [b] at (0.0, 0.0) {  } 
child {node [b] at (0.0, 1.0) {  }  
}
;
}}

\newcommand{\forestS}{
\tikz[planar forest] {

\node [b] at (0.0, 0.0) {  } 
child {node [b] at (-0.5, 1.0) {  }  
}
child {node [b] at (0.5, 1.0) {  }  
child {node [b] at (0.0, 1.0) {  }  
}
}
;
}}

\newcommand{\forestT}{
\tikz[planar forest] {

\node [b] at (0.0, 0.0) {  } 
;
}}

\newcommand{\forestU}{
\tikz[planar forest] {

\node [b] at (0.0, 0.0) {  } 
child {node [b] at (0.0, 1.0) {  }  
}
;
}}

\newcommand{\forestV}{
\tikz[planar forest] {

\node [b] at (0.0, 0.0) {  } 
child {node [b] at (-0.5, 1.0) {  }  
}
child {node [b] at (0.5, 1.0) {  }  
}
;
}}

\newcommand{\forestW}{
\tikz[planar forest] {

\node [b] at (0.0, 0.0) {  } 
child {node [b] at (-1.0, 1.0) {  }  
}
child {node [b] at (0.0, 1.0) {  }  
}
child {node [b] at (1.0, 1.0) {  }  
}
;
}}

\newcommand{\forestX}{
\tikz[planar forest] {

\node [b] at (0.0, 0.0) {  } 
child {node [b] at (-1.5, 1.0) {  }  
}
child {node [b] at (-0.5, 1.0) {  }  
}
child {node [b] at (0.5, 1.0) {  }  
}
child {node [b] at (1.5, 1.0) {  }  
}
;
}}

\newcommand{\forestY}{
\tikz[planar forest] {

\node [b] at (0.0, 0.0) {  } 
child {node [b] at (0.0, 1.0) {  }  
child {node [b] at (0.0, 1.0) {  }  
}
}
;
}}

\newcommand{\forestAB}{
\tikz[planar forest] {

\node [b] at (0.0, 0.0) {  } 
child {node [b] at (0.0, 1.0) {  }  
child {node [b] at (-0.5, 1.0) {  }  
}
child {node [b] at (0.5, 1.0) {  }  
}
}
;
}}

\newcommand{\forestBB}{
\tikz[planar forest] {

\node [b] at (0.0, 0.0) {  } 
child {node [b] at (-0.5, 1.0) {  }  
}
child {node [b] at (0.5, 1.0) {  }  
child {node [b] at (0.0, 1.0) {  }  
}
}
;
}}

\newcommand{\forestCB}{
\tikz[planar forest] {

\node [b] at (0.0, 0.0) {  } 
child {node [b] at (-1.0, 1.0) {  }  
child {node [b] at (0.0, 1.0) {  }  
}
}
child {node [b] at (0.0, 1.0) {  }  
}
child {node [b] at (1.0, 1.0) {  }  
}
;
}}

\newcommand{\forestDB}{
\tikz[planar forest] {

\node [b] at (0.0, 0.0) {  } 
child {node [b] at (-0.5, 1.0) {  }  
child {node [b] at (0.0, 1.0) {  }  
}
}
child {node [b] at (0.5, 1.0) {  }  
child {node [b] at (0.0, 1.0) {  }  
}
}
;
}}

\newcommand{\forestEB}{
\tikz[planar forest] {

\node [b] at (0.0, 0.0) {  } 
child {node [b] at (0.0, 1.0) {  }  
child {node [b] at (-1.0, 1.0) {  }  
}
child {node [b] at (0.0, 1.0) {  }  
}
child {node [b] at (1.0, 1.0) {  }  
}
}
;
}}

\newcommand{\forestFB}{
\tikz[planar forest] {

\node [b] at (0.0, 0.0) {  } 
child {node [b] at (-0.5, 1.0) {  }  
child {node [b] at (-0.5, 1.0) {  }  
}
child {node [b] at (0.5, 1.0) {  }  
}
}
child {node [b] at (0.5, 1.0) {  }  
}
;
}}

\newcommand{\forestGB}{
\tikz[planar forest] {

\node [b] at (0.0, 0.0) {  } 
child {node [b] at (0.0, 1.0) {  }  
child {node [b] at (0.0, 1.0) {  }  
child {node [b] at (0.0, 1.0) {  }  
}
}
}
;
}}

\newcommand{\forestHB}{
\tikz[planar forest] {

\node [b] at (0.0, 0.0) {  } 
child {node [b] at (0.0, 1.0) {  }  
child {node [b] at (0.0, 1.0) {  }  
child {node [b] at (-0.5, 1.0) {  }  
}
child {node [b] at (0.5, 1.0) {  }  
}
}
}
;
}}

\newcommand{\forestIB}{
\tikz[planar forest] {

\node [b] at (0.0, 0.0) {  } 
child {node [b] at (0.0, 1.0) {  }  
child {node [b] at (-0.5, 1.0) {  }  
child {node [b] at (0.0, 1.0) {  }  
}
}
child {node [b] at (0.5, 1.0) {  }  
}
}
;
}}

\newcommand{\forestJB}{
\tikz[planar forest] {

\node [b] at (0.0, 0.0) {  } 
child {node [b] at (-0.5, 1.0) {  }  
child {node [b] at (0.0, 1.0) {  }  
child {node [b] at (0.0, 1.0) {  }  
}
}
}
child {node [b] at (0.5, 1.0) {  }  
}
;
}}

\newcommand{\forestKB}{
\tikz[planar forest] {

\node [b] at (0.0, 0.0) {  } 
child {node [b] at (0.0, 1.0) {  }  
child {node [b] at (0.0, 1.0) {  }  
child {node [b] at (0.0, 1.0) {  }  
child {node [b] at (0.0, 1.0) {  }  
}
}
}
}
;
}}

\newcommand{\forestLB}{
\tikz[planar forest] {

\node [b] at (0.0, 0.0) {  } 
child {node [b] at (0.0, 1.0) {  }  
}
;
}}

\newcommand{\forestMB}{
\tikz[planar forest] {

\node [b] at (0.0, 0.0) {  } 
child {node [b] at (0.0, 1.0) {  }  
}
;
}}

\newcommand{\forestNB}{
\tikz[planar forest] {

\node [b] at (0.0, 0.0) {  } 
child {node [b] at (-0.5, 1.0) {  }  
}
child {node [b] at (0.5, 1.0) {  }  
}
;
}}

\newcommand{\forestOB}{
\tikz[planar forest] {

\node [b] at (0.0, 0.0) {  } 
child {node [b] at (-0.5, 1.0) {  }  
}
child {node [b] at (0.5, 1.0) {  }  
}
;
}}

\newcommand{\forestPB}{
\tikz[planar forest] {

\node [b] at (0.0, 0.0) {  } 
child {node [b] at (0.0, 1.0) {  }  
child {node [b] at (0.0, 1.0) {  }  
}
}
;
}}

\newcommand{\forestQB}{
\tikz[planar forest] {

\node [b] at (0.0, 0.0) {  } 
child {node [b] at (0.0, 1.0) {  }  
child {node [b] at (0.0, 1.0) {  }  
}
}
;
}}

\newcommand{\forestRB}{
\tikz[planar forest] {

\node [b] at (0.0, 0.0) {  } 
child {node [b] at (-1.0, 1.0) {  }  
}
child {node [l] at (0.0, 1.0) { $\cdot \cdot \cdot$ }  
}
child {node [b] at (1.0, 1.0) {  }  
}
;
}}

\newcommand{\forestSB}{
\tikz[planar forest] {

\node [b] at (0.0, 0.0) {  } 
child {node [b] at (-1.0, 1.0) {  }  
}
child {node [l] at (0.0, 1.0) { $\cdot \cdot \cdot$ }  
}
child {node [b] at (1.0, 1.0) {  }  
}
;
}}

\newcommand{\forestTB}{
\tikz[planar forest] {

\node [a] at (0.0, 0.0) {  } 
child {node [b] at (0.0, 1.0) {  }  
child {node [b] at (-1.0, 1.0) {  }  
}
child {node [l] at (0.0, 1.0) { $\cdot \cdot \cdot$ }  
}
child {node [b] at (1.0, 1.0) {  }  
}
}
;
}}

\newcommand{\forestUB}{
\tikz[planar forest] {

\node [a] at (0.0, 0.0) {  } 
child {node [b] at (-1.5, 1.0) {  }  
}
child {node [b] at (-0.5, 1.0) {  }  
}
child {node [l] at (0.5, 1.0) { $\cdot \cdot \cdot$ }  
}
child {node [b] at (1.5, 1.0) {  }  
}
;
}}

\newcommand{\forestVB}{
\tikz[planar forest] {

\node [a] at (0.0, 0.0) {  } 
;
}}

\newcommand{\forestWB}{
\tikz[planar forest] {

\node [b] at (0.0, 0.0) {  } 
child {node [b] at (-1.5, 1.0) {  }  
}
child {node [l] at (-0.5, 1.0) { $\cdot \cdot \cdot$ }  
}
child {node [b] at (0.5, 1.0) {  }  
}
child {node [b] at (1.5, 1.0) {  }  
child {node [a] at (0.0, 1.0) {  }  
}
}
;
}}

\newcommand{\forestXB}{
\tikz[planar forest] {

\node [b] at (0.0, 0.0) {  } 
child {node [a] at (0.0, 1.0) {  }  
}
;
}}

\newcommand{\forestYB}{
\tikz[planar forest] {

\node [b] at (0.0, 0.0) {  } 
child {node [b] at (-2.0, 1.0) {  }  
}
child {node [l] at (-1.0, 1.0) { $\cdot \cdot \cdot$ }  
}
child {node [b] at (0.0, 1.0) {  }  
}
child {node [b] at (1.0, 1.0) {  }  
}
child {node [a] at (2.0, 1.0) {  }  
}
;
}}

\newcommand{\forestAC}{
\tikz[planar forest] {

\node [a] at (0.0, 0.0) {  } 
;
}}

\newcommand{\forestBC}{
\tikz[planar forest] {

\node [b] at (0.0, 0.0) {  } 
;
}}

\newcommand{\forestCC}{
\tikz[planar forest] {

\node [b] at (0.0, 0.0) {  } 
child {node [b] at (0.0, 1.0) {  }  
}
;
}}

\newcommand{\forestDC}{
\tikz[planar forest] {

\node [b] at (0.0, 0.0) {  } 
child {node [b] at (0.0, 1.0) {  }  
child {node [b] at (0.0, 1.0) {  }  
}
}
;
}}

\newcommand{\forestEC}{
\tikz[planar forest] {

\node [b] at (0.0, 0.0) {  } 
child {node [b] at (0.0, 1.0) {  }  
child {node [b] at (0.0, 1.0) {  }  
child {node [b] at (0.0, 1.0) {  }  
}
}
}
;
}}

\newcommand{\forestFC}{
\tikz[planar forest] {

\node [b] at (0.0, 0.0) {  } 
child {node [b] at (0.0, 1.0) {  }  
child {node [b] at (0.0, 1.0) {  }  
child {node [b] at (0.0, 1.0) {  }  
child {node [b] at (0.0, 1.0) {  }  
}
}
}
}
;
}}

\newcommand{\forestGC}{
\tikz[planar forest] {

\node [b] at (0.0, 0.0) {  } 
child {node [l] at (0.0, 1.0) { $\cdot \cdot \cdot$ }  
child {node [b] at (0.0, 1.0) {  }  
}
}
;
}}

\newcommand{\forestHC}{
\tikz[planar forest] {

\node [b] at (0.0, 0.0) {  } 
child {node [l] at (0.0, 1.0) { $\cdot \cdot \cdot$ }  
child {node [b] at (0.0, 1.0) {  }  
}
}
;
}}

\newcommand{\forestIC}{
\tikz[planar forest] {

\node [b] at (0.0, 0.0) {  } 
child {node [l] at (0.0, 1.0) { $\cdot \cdot \cdot$ }  
child {node [b] at (0.0, 1.0) {  }  
}
}
;
}}

\newcommand{\forestJC}{
\tikz[planar forest] {

\node [b] at (0.0, 0.0) {  } 
child {node [l] at (0.0, 1.0) { $\cdot \cdot \cdot$ }  
child {node [b] at (0.0, 1.0) {  }  
}
}
;
}}

\newcommand{\forestKC}{
\tikz[planar forest] {

\node [b] at (0.0, 0.0) {  } 
child {node [l] at (0.0, 1.0) { $\cdot \cdot \cdot$ }  
child {node [b] at (0.0, 1.0) {  }  
}
}
;
}}

\newcommand{\forestLC}{
\tikz[planar forest] {

\node [b] at (0.0, 0.0) {  } 
child {node [l] at (0.0, 1.0) { $\cdot \cdot \cdot$ }  
child {node [b] at (0.0, 1.0) {  }  
}
}
;
}}

\newcommand{\forestMC}{
\tikz[planar forest] {

\node [b] at (0.0, 0.0) {  } 
child {node [l] at (0.0, 1.0) { $\cdot \cdot \cdot$ }  
child {node [b] at (0.0, 1.0) {  }  
}
}
;
}}

\newcommand{\forestNC}{
\tikz[planar forest] {

\node [b] at (0.0, 0.0) {  } 
child {node [l] at (0.0, 1.0) { $\cdot \cdot \cdot$ }  
child {node [b] at (0.0, 1.0) {  }  
}
}
;
}}

\newcommand{\forestOC}{
\tikz[planar forest] {

\node [b] at (0.0, 0.0) {  } 
child {node [l] at (0.0, 1.0) { $\cdot \cdot \cdot$ }  
child {node [b] at (0.0, 1.0) {  }  
}
}
;
}}

\newcommand{\forestPC}{
\tikz[planar forest] {

\node [b] at (0.0, 0.0) {  } 
child {node [l] at (0.0, 1.0) { $\cdot \cdot \cdot$ }  
child {node [b] at (0.0, 1.0) {  }  
}
}
;
}}

\newcommand{\forestQC}{
\tikz[planar forest] {

\node [b] at (0.0, 0.0) {  } 
child {node [l] at (0.0, 1.0) { $\cdot \cdot \cdot$ }  
child {node [b] at (0.0, 1.0) {  }  
}
}
;
}}

\newcommand{\forestRC}{
\tikz[planar forest] {

\node [b] at (0.0, 0.0) {  } 
child {node [l] at (0.0, 1.0) { $\cdot \cdot \cdot$ }  
child {node [b] at (0.0, 1.0) {  }  
}
}
;
}}

\newcommand{\forestSC}{
\tikz[planar forest] {

\node [b] at (0.0, 0.0) {  } 
child {node [l] at (0.0, 1.0) { $\cdot \cdot \cdot$ }  
child {node [b] at (0.0, 1.0) {  }  
}
}
;
}}

\newcommand{\forestTC}{
\tikz[planar forest] {

\node [b] at (0.0, 0.0) {  } 
child {node [l] at (0.0, 1.0) { $\cdot \cdot \cdot$ }  
child {node [b] at (0.0, 1.0) {  }  
}
}
;
}}

\newcommand{\forestUC}{
\tikz[planar forest] {

\node [b] at (0.0, 0.0) {  } 
child {node [l] at (0.0, 1.0) { $\cdot \cdot \cdot$ }  
child {node [b] at (0.0, 1.0) {  }  
}
}
;
}}

\newcommand{\forestVC}{
\tikz[planar forest] {

\node [b] at (0.0, 0.0) {  } 
child {node [l] at (0.0, 1.0) { $\cdot \cdot \cdot$ }  
child {node [b] at (0.0, 1.0) {  }  
}
}
;
}}

\newcommand{\forestWC}{
\tikz[planar forest] {

\node [b] at (0.0, 0.0) {  } 
child {node [l] at (0.0, 1.0) { $\cdot \cdot \cdot$ }  
child {node [b] at (0.0, 1.0) {  }  
}
}
;
}}

\newcommand{\forestXC}{
\tikz[planar forest] {

\node [b] at (0.0, 0.0) {  } 
child {node [l] at (0.0, 1.0) { $\cdot \cdot \cdot$ }  
child {node [b] at (0.0, 1.0) {  }  
}
}
;
}}

\newcommand{\forestYC}{
\tikz[planar forest] {

\node [b] at (0.0, 0.0) {  } 
child {node [l] at (0.0, 1.0) { $\cdot \cdot \cdot$ }  
child {node [b] at (0.0, 1.0) {  }  
}
}
;
}}

\newcommand{\forestAD}{
\tikz[planar forest] {

\node [b] at (0.0, 0.0) {  } 
child {node [l] at (0.0, 1.0) { $\cdot \cdot \cdot$ }  
child {node [b] at (0.0, 1.0) {  }  
}
}
;
}}

\newcommand{\forestBD}{
\tikz[planar forest] {

\node [b] at (0.0, 0.0) {  } 
child {node [b] at (0.0, 1.0) {  }  
}
;
}}

\newcommand{\forestCD}{
\tikz[planar forest] {

\node [b] at (0.0, 0.0) {  } 
child {node [b] at (0.0, 1.0) {  }  
}
;
}}

\newcommand{\forestDD}{
\tikz[planar forest] {

\node [b] at (0.0, 0.0) {  } 
child {node [b] at (0.0, 1.0) {  }  
}
;
}}

\newcommand{\forestED}{
\tikz[planar forest] {

\node [b] at (0.0, 0.0) {  } 
;
}}

\newcommand{\forestFD}{
\tikz[planar forest] {

\node [b] at (0.0, 0.0) {  } 
;
}}

\newcommand{\forestGD}{
\tikz[planar forest] {

\node [b] at (0.0, 0.0) {  } 
;
}}

\newcommand{\forestHD}{
\tikz[planar forest] {

\node [b] at (0.0, 0.0) {  } 
;
}}

\newcommand{\forestID}{
\tikz[planar forest] {

\node [b] at (0.0, 0.0) {  } 
;
}}

\newcommand{\forestJD}{
\tikz[planar forest] {

\node [b] at (0.0, 0.0) {  } 
;
}}

\newcommand{\forestKD}{
\tikz[planar forest] {

\node [b] at (0.0, 0.0) {  } 
child {node [b] at (0.0, 1.0) {  }  
}
;
}}

\newcommand{\forestLD}{
\tikz[planar forest] {

\node [b] at (0.0, 0.0) {  } 
child {node [b] at (0.0, 1.0) {  }  
child {node [b] at (0.0, 1.0) {  }  
}
}
;
}}

\newcommand{\forestMD}{
\tikz[planar forest] {

\node [b] at (0.0, 0.0) {  } 
child {node [b] at (0.0, 1.0) {  }  
}
;
}}

\newcommand{\forestND}{
\tikz[planar forest] {

\node [b] at (0.0, 0.0) {  } 
child {node [b] at (0.0, 1.0) {  }  
child {node [b] at (0.0, 1.0) {  }  
}
}
;
}}

\newcommand{\forestOD}{
\tikz[planar forest] {

\node [b] at (0.0, 0.0) {  } 
child {node [b] at (0.0, 1.0) {  }  
child {node [b] at (0.0, 1.0) {  }  
}
}
;
}}

\newcommand{\forestPD}{
\tikz[planar forest] {

\node [b] at (0.0, 0.0) {  } 
child {node [b] at (0.0, 1.0) {  }  
child {node [b] at (0.0, 1.0) {  }  
}
}
;
}}

\newcommand{\forestQD}{
\tikz[planar forest] {

\node [b] at (0.0, 0.0) {  } 
child {node [b] at (0.0, 1.0) {  }  
child {node [b] at (0.0, 1.0) {  }  
}
}
;
}}

\newcommand{\forestRD}{
\tikz[planar forest] {

\node [b] at (0.0, 0.0) {  } 
child {node [b] at (0.0, 1.0) {  }  
}
;
}}

\newcommand{\forestSD}{
\tikz[planar forest] {

\node [b] at (0.0, 0.0) {  } 
child {node [b] at (0.0, 1.0) {  }  
}
;
}}

\newcommand{\forestTD}{
\tikz[planar forest] {

\node [b] at (0.0, 0.0) {  } 
child {node [b] at (0.0, 1.0) {  }  
}
;
}}

\newcommand{\forestUD}{
\tikz[planar forest] {

\node [b] at (0.0, 0.0) {  } 
child {node [b] at (0.0, 1.0) {  }  
}
;
}}

\newcommand{\forestVD}{
\tikz[planar forest] {

\node [b] at (0.0, 0.0) {  } 
child {node [b] at (0.0, 1.0) {  }  
child {node [b] at (0.0, 1.0) {  }  
child {node [b] at (0.0, 1.0) {  }  
}
}
}
;
}}

\newcommand{\forestWD}{
\tikz[planar forest] {

\node [b] at (0.0, 0.0) {  } 
child {node [b] at (0.0, 1.0) {  }  
child {node [b] at (0.0, 1.0) {  }  
}
}
;
}}

\newcommand{\forestXD}{
\tikz[planar forest] {

\node [b] at (0.0, 0.0) {  } 
child {node [b] at (0.0, 1.0) {  }  
child {node [b] at (0.0, 1.0) {  }  
}
}
;
}}

\newcommand{\forestYD}{
\tikz[planar forest] {

\node [b] at (0.0, 0.0) {  } 
child {node [b] at (0.0, 1.0) {  }  
child {node [b] at (0.0, 1.0) {  }  
}
}
;
}}

\newcommand{\forestAE}{
\tikz[planar forest] {

\node [b] at (0.0, 0.0) {  } 
child {node [b] at (0.0, 1.0) {  }  
child {node [b] at (0.0, 1.0) {  }  
}
}
;
}}

\newcommand{\forestBE}{
\tikz[planar forest] {

\node [b] at (0.0, 0.0) {  } 
child {node [b] at (0.0, 1.0) {  }  
child {node [b] at (0.0, 1.0) {  }  
}
}
;
}}

\newcommand{\forestCE}{
\tikz[planar forest] {

\node [b] at (0.0, 0.0) {  } 
child {node [b] at (0.0, 1.0) {  }  
}
;
}}

\newcommand{\forestDE}{
\tikz[planar forest] {

\node [b] at (0.0, 0.0) {  } 
child {node [b] at (0.0, 1.0) {  }  
}
;
}}

\newcommand{\forestEE}{
\tikz[planar forest] {

\node [b] at (0.0, 0.0) {  } 
child {node [b] at (0.0, 1.0) {  }  
}
;
}}

\newcommand{\forestFE}{
\tikz[planar forest] {

\node [b] at (0.0, 0.0) {  } 
child {node [b] at (0.0, 1.0) {  }  
child {node [b] at (0.0, 1.0) {  }  
child {node [b] at (0.0, 1.0) {  }  
}
}
}
;
}}

\newcommand{\forestGE}{
\tikz[planar forest] {

\node [b] at (0.0, 0.0) {  } 
child {node [b] at (0.0, 1.0) {  }  
child {node [b] at (0.0, 1.0) {  }  
child {node [b] at (0.0, 1.0) {  }  
}
}
}
;
}}

\usepackage{tikz}
\usepackage{tikz-cd}
\usepackage{pgfplots}

\definecolor{C0}{HTML}{1F77B4}
\definecolor{C0L}{HTML}{AEC7E8}
\definecolor{C1}{HTML}{FF7F0E}
\definecolor{C1L}{HTML}{FFBB78}
\definecolor{C2}{HTML}{2CA02C}
\definecolor{C2L}{HTML}{98DF8A}
\definecolor{C3}{HTML}{D62728}
\definecolor{C3L}{HTML}{FF9896}
\definecolor{C4}{HTML}{9467BD}
\definecolor{C4L}{HTML}{C5B0D5}
\definecolor{C5}{HTML}{8C564B}
\definecolor{C5L}{HTML}{C49C94}
\definecolor{C6}{HTML}{E377C2}
\definecolor{C6L}{HTML}{F7B6D2}
\definecolor{C7}{HTML}{7F7F7F}
\definecolor{C7L}{HTML}{C7C7C7}
\definecolor{C8}{HTML}{BCBD22}
\definecolor{C8L}{HTML}{DBDB8D}
\definecolor{C9}{HTML}{17BECF}
\definecolor{C9L}{HTML}{9EDAE5}

\newcommand*\diff{\mathop{}\!\mathrm{d}}

\newcommand{\reva}[1]{#1}
\newcommand{\revb}[1]{#1}
\newcommand{\revc}[1]{#1}

\newcommand{\coloneq}{\mathrel{\mathop:}=}

\DeclareMathOperator{\diag}{diag}

\theoremstyle{thmstyleone}%
\newtheorem{theorem}{Theorem}
\newtheorem{corollary}{Corollary}
\newtheorem{lemma}{Lemma}
\newtheorem{proposition}[theorem]{Proposition}%

\theoremstyle{thmstyletwo}%
\newtheorem{example}{Example}%
\newtheorem{remark}{Remark}%

\theoremstyle{thmstylethree}%
\newtheorem{definition}{Definition}%

  \DeclareMathOperator*{\argmin}{argmin}

\begin{document}

\title[\revc{Order jumps in Spectral Deferred Corrections}]{\revc{Order jumps in Spectral Deferred Corrections}}


\author*[1]{\fnm{Eugen} \sur{Bronasco}}\email{bronasco@chalmers.se}
\equalcont{These authors contributed equally to this work.}

\author*[2]{\fnm{Joscha} \sur{Fregin}}\email{joscha.fregin@tuhh.de}
\equalcont{These authors contributed equally to this work.}

\author[2]{\fnm{Daniel} \sur{Ruprecht}}\email{ruprecht@tuhh.de}

\author[3]{\fnm{Gilles} \sur{Vilmart}}\email{gilles.vilmart@unige.ch}

\affil*[1]{\orgdiv{Department of Mathematical Sciences}, \orgname{Chalmers University of Technology and University of Gothenburg}, \orgaddress{\street{Chalmersplatsen 4}, \city{Gothenburg}, \postcode{SE-412 96}, \state{Gothenburg}, \country{Sweden}}}

\affil[2]{\orgdiv{Chair Computational Mathematics, Institute of Mathematics}, \orgname{Hamburg University of Technology}, \orgaddress{\street{Am Schwarzenberg-Campus 3}, \city{Hamburg}, \postcode{21073}, \state{Hamburg}, \country{Germany}}}

\affil*[3]{\orgdiv{Section de mathématiques}, \orgname{Université de Genève}, \orgaddress{\street{CP 64}, \city{Genève 4}, \postcode{1211}, \state{Genève}, \country{Switzerland}}}



\abstract{
 \revc{Spectral deferred corrections (SDC) are time-stepping schemes were a low-order integrator is used iteratively to produce a higher order numerical approximation. 
SDC methods are well-known to obtain at least one additional order of accuracy per iteration.
 We analyse the phenomenon of ``order jumps'', which were observed experimentally in previous work, where the order increases by two in a single iteration.
Our analysis relies on Butcher series for the analysis of Runge--Kutta methods (RKM) and includes the case of collocation and discontinuous collocation RKM.}
We prove that order jumps can be obtained by using appropriate inconsistent, implicit, parallelisable error discretisations.
\reva{We also investigate the stability properties of order-jump SDC methods and show that, in general, they reduce to that of explicit RKM. 
Suitable combinations of error discretisations, informed by our analysis, could balance order gains and stability and deliver improved SDC methods in the future.}
}




\maketitle

\section{Introduction}\label{sec:intro} 

Spectral deferred corrections (SDC) have received significant amounts of attention since their introduction \citep{DuttEtAl2000}. 
Many different variants have emerged, including semi-implicit SDC~\citep{Minion2003}, fast-wave slow-wave SDC~\citep{RuprechtSpeck2016}, GMRES-SDC~\citep{HuangEtAl2006}, parallel SDC~\citep{Speck2017}, low storage SDC~\citep{Crockatt2018}, SDC with penalization of the correction term~\citep{ShuEtAl2007}, fast converging SDC~\citep{Weiser2014}, SDC for second order initial value problems~\citep{WinkelEtAl2015} and Integral Deferred Corrections (IDC)~\citep{ChristliebEtAl2009}. 
The name IDC is sometimes used equivalently to SDC, but, in~\citep{ChristliebEtAl2009}, IDC describes a deferred correction variant that uses higher order integrators for the corrections and introduces new stages in between existing ones. 
For an overview of SDC methods see \citep{OngEtAl2020sdc}.

\revb{SDC is useful for at least moderately stiff problems, since its purely explicit variants are typically not competitive with respect to computational efficiency~\cite{Ketcheson2014}.
Its main advantage is its iterative nature, which allows to construct high-order integrators with special features like semi-implicit splitting~\cite{Minion2003}, multi-rate time-stepping~\cite{NaumannEtAl2018} or parallelism~\cite{Speck2018}.
Several works show computational gains from tailored variants of SDC for different applications involving both ordinary and partial differential equations,  at least if reasonably high accuracy is required, e.g. for simulations of fast ions in fusion reactors~\cite{TretiakEtAl2021}, geo-physical flows~\cite{BrownEtAl2025} or computational fluid dynamics~\cite{guesmi2023assessment}.
}

There are at least four critical design choices during the construction of an SDC method. 
\begin{enumerate}
 \item The number and type of quadrature nodes of the underlying collocation method (e.g. Gauss, equidistant or some other set of nodes).
 \item The procedure to calculate the initial guesses at the quadrature nodes that SDC requires to start iterating (e.g. simply copying the value brought forward from the previous time step or performing one initial prediction using the sweeper, see next item).
 \item The type of sweeper (also sometimes called preconditioner) that updates SDC's stages in every iteration.
To construct the sweeper, there exist two main approaches. 
Originally, SDC used a discretization consistent with a differential equation for the error~\citep{DuttEtAl2000,ChristliebEtAl2009}. 
However, taking an algebraic perspective of SDC, it was later shown that inconsistent discretizations (that do not correspond to a consistent discretisation of the error equation) can be used to improve convergence~\citep{Weiser2014}. We will refer to the sweeper or \emph{error equation discretization} as EED. Note that some variants of SDC use an EED that changes in each iteration~\citep{Weiser2014, CaklovicEtAl2025}.
\item The last step of SDC is to calculate the solution at the endpoint of the timestep from the stage values. 
Extrapolation is proposed in~\citep{DuttEtAl2000} where the interpolating polynomial defined by the stage values is evaluated. 
An alternative is to use collocation nodes where the last stage is equal to the end point and to simply copy the solution.
Quadrature uses the weights provided by the underlying quadrature rule and applies them to function evaluations of the stages as is usually done in Runge-Kutta methods (RKM).
\end{enumerate}

Links between SDC and RKM have already been established. 
A relation between IDC and RKM is shown by Christlieb et al.~\citep{ChristliebEtAl2009}.
Butcher series were used to study the convergence of general deferred correction schemes~\citep{HairerOID78} even before the introduction of SDC..
For selecting SDC methods, a Shu-Osher form has been described~\citep{Gottlieb2009,Gottlieb2011}. 
The structure of Butcher tableaus that emerges from this approach has first been described in~\citep{VanDerHouwen1991,Houwen1992,Houwen1993} which predates the original introduction of SDC in~\citep{DuttEtAl2000} by almost a decade. 
Butcher tableaus for IMEX-SDC have been derived in~\citep{BoscarinoEtAl2018}.

The present paper shows how to incorporate all those SDC variants in the RKM framework. 
We derive SDC from the Picard iteration perspective and prove, using Butcher series, that all SDC variants with $K$ iterations have at least order $K$. 
Importantly, our proof shows that this is independent of the chosen nodes or error equation discretisation (EED). 
Our theory also shows that, in some cases, we can guarantee a gain of two orders of convergence per iteration for parallelizable EEDs.
This was observed numerically~\citep{CaklovicEtAl2025} but we provide a rigorous proof.
We also show that initialising SDC with a given EED is equivalent to initializing SDC by copying the initial value to all nodes and performing one iteration using the EED.

\paragraph{Related work on the theory of SDC.}
Most convergence results for SDC consider either equidistant nodes and consistent EED~\citep{HairerOID78,HansenStrain2006,HansenStrain2011,CausleySeal2019,ChristliebEtAl2010_MoC,ChristliebEtAl2009} or arbitrary nodes and inconsistent EED~\citep{HagstromZhou2006,HuangEtAl2006,ShuEtAl2007,VanDerHouwen1991}. 
Few works address the combination of arbitrary nodes and consistent EEDs, one example being~\citep{RuprechtSpeck2016}.
Consistent, high order EED are also discussed in \citep{HairerOID78,HansenStrain2006,HansenStrain2011,ChristliebEtAl2010_MoC,ChristliebEtAl2009,TangEtAl2013}.

SDC can be initialised using a low order method~\citep{DuttEtAl2000} or by copying the initial state~\citep{Houwen1992,RuprechtSpeck2016}.
Using quadrature to obtain the end point can increase the order of the method by one if the order of the underlying method is not yet reached~\citep{ShuEtAl2007}. 
Arbitrary order one-step and multi-step methods for equidistant nodes as EEDs for SDC were analyzed in~\citep{HansenStrain2011,HansenStrain2006}.
They find that, in contrast to non-equidistant nodes, the order of SDC increases by the order of the EED scheme with each iteration.
They also seem to be the first to suggest to change the EED in every iteration, which is later found necessary to construct parallel SDC variants that converge robustly for stiff problems~\citep{CaklovicEtAl2025}.

\citep{HagstromZhou2006} consider non-equidistant grids and show that the maximum order is bounded by the underlying method and not the number of nodes. 
Their findings hold also hold for inconsistent EED. 
\citep{HuangEtAl2006} prove that GMRES-SDC with $s$ Gauss nodes can achieve order $2s$. 
A convergence proof for scaled EED, i.e. SDC where the coefficient matrix for the error discretisation is scaled by a constant factor,  can be found in \citep{ShuEtAl2007}. 
Examples of efficient inconsistent EED's are given in~\citep{Weiser2014, Speck2017}.
\citep{CausleySeal2019} consider SDC methods where the error is discretised either with implicit, explicit Euler or implicit trapezoidal rule. 
They prove that SDC's order increases by one per iteration for Euler or two for the trapezoidal rule and also find that the error equation does not need to be discretised with a consistent method. 
\citep{ChristliebEtAl2010_MoC} discretise the error equation using arbitrary order RKM to construct consistent EED when using equidistant nodes.

\citep{ChristliebEtAl2009} expand the IDC methods introduced in~\citep{ChristliebEtAl2010_MoC} by allowing multistep methods to discretise the error equation. 
They cast their IDC as RKM and build Butcher tableaux to find the local truncation error and show that, for non-equidistant nodes, a high order ($p \geq 2$) EED does not necessarily increase the order of SDC by more than one per iteration. \citep{HairerOID78} interprets deferred correction methods with $n$ equidistant nodes as $n$ steps of an integrator and use the theory of Butcher series to prove the order increase per iteration.

\revb{Already in 1996, \citet{jackson1996analysis} used Butcher series to show that, when a fixed-point iteration is used to solve for the stages of an implicit RKM, so-called \emph{iterated RKM},  the order of the overall method increases by one per iteration.
This is, in essence, an SDC iteration without an EED and therefore our analysis can be seen as a generalization of their result to general EEDs.}
\citet{RuprechtSpeck2016} prove that SDC gains one order per iteration when using implicit-explicit methods as EED, but the prove works for standard SDC as well. 
Their proof breaks down for inconsistent EED because their bound on the supremums norm of the matrix containing the collocation weights is no longer guaranteed to be true.
Finally, \citep{VanDerHouwen1991} give a proof that after $k$ iterations SDC yields at least order $p = \text{min}(k,r)$, where $r$ is the minimal stage order of the underlying collocation method.

\paragraph{Order jumps in SDC.}
An \emph{order jump} refers to a phenomenon in SDC methods where the order of accuracy of the solution increases by more than one in a single iteration.
Such order jumps have been observed in previous works~\citep{HansenStrain2011,HansenStrain2006,ChristliebEtAl2009,ChristliebEtAl2010_MoC}. 
They prove that order jumps by order $p$ occur, when the error equation is discretised using a $p$ order method and is applied within an SDC method using equidistant node distributions (see e.g. Lemma 3.9, Lemma 3.10 in \citep{ChristliebEtAl2009}).
Order jumps with non-equidistant grids are briefly discussed by Christlieb et al.~\citep{ChristliebEtAl2009} (Example 3.13) and were observed numerically for parallel SDC methods~\citep{CaklovicEtAl2025} where the diagonal EED called MIN-SR-NS is introduced.
\reva{\citet{Bolten2026} observe order jumps when applying SDC to differential-algebraic equations and explicitly acknowledge that ``the occurrence of order jumps cannot be explained, and no prediction is
yet possible for them''.}
\reva{We illustrate the phenomenon in Tables~\ref{table:MINRADAU} and~\ref{table:JUMPERRADAU}.
Table~\ref{table:MINRADAU} show the theoretical convergence order of SDC using the MIN-SR-NS diagonal EED by Caklovic et al~\cite{CaklovicEtAl2025} while Table~\ref{table:JUMPERRADAU} shows theoretical convergence order for the order-jump EED characterized later in Theorem~\ref{thm:order_jump}.
Convergence orders are computed by assembling the Butcher table for each SDC method, as shown below in~\eqref{eq:SDCButcherTableau}, and confirming algebraic order conditions using the \texttt{Julia} framework by Ketcheson et al.~\cite{Ketcheson2023}.
Rows indicate different number $s$ of RadauIIA nodes while columns indicate the number of SDC iterations $k$.
Order jumps, increases of convergence order by two with a single iteration, are marked in bold.
While the diagonal EED only shows a single order jump in iteration $k = s-1$, the order-jump EED increases order by two in every iteration, up to the order $2s-1$ of the underlying collocation method.
To our knowledge, no theory exists that predicts order jumps for non-equidistant nodes for general EEDs.
}
\begin{table}[h]
	\bgroup
	\def\arraystretch{1}%
	\begin{tabular}{c||c|c|c|c|c|c|c|c|c|c|c|c|c|c|c|}
		\centering
		 $s$ $\backslash$ $k$ & 1 & 2 & 3 & 4 & 5 & 6 & 7 & 8 & 9 & 10 & 11 & 12 & 13 & 14 & 15 \\
		\hline
		\hline
		1 & 1 & 1 & 1 & 1 & 1 & 1 & 1 & 1 & 1 & 1 & 1 & 1 & 1 & 1 & 1 \\
		\hline
		2 & $\bm{2}$ & 3 & 3 & 3 & 3 & 3 & 3 & 3 & 3 & 3 & 3 & 3 & 3 &3 & 3 \\
		\hline
		3 & 1 & $\bm{3}$ & 4 & 5 & 5 & 5 & 5 & 5 & 5 & 5 & 5 & 5 & 5 & 5 & 5\\
		\hline
		4 & 1 & 2 & $\bm{4}$ & 5 & 6 & 7 & 7 & 7 & 7 & 7 & 7 & 7 & 7 & 7 & 7 \\
		\hline
		5 & 1 & 2 & 3 & $\bm{5}$ & 6 & 7 & 8 & 9 & 9 & 9 & 9 & 9 & 9 & 9 & 9  \\
		\hline
		6 & 1 & 2 & 3 & 4 & $\bm{6}$ & 7 & 8 & 9 &10&11&11&11&11&11&11 \\
		\hline
		7 & 1 & 2 & 3 & 4 & 5 & $\bm{7}$ & 8 & 9 &10&11&12&13&13&13&13 \\
		\hline
		8 & 1 & 2 & 3 & 4 & 5 & 6 & $\bm{8}$ & 9 &10&11&12&13&14&15&15 \\
		\hline
	\end{tabular}
	\egroup
	\caption{Convergence order of SDC using the MIN-SR-NS EED by Caklovic et al.~\cite{CaklovicEtAl2025}. The method uses $s$ Radau IIA nodes (lines) for a maximum order of $2s -1$ and $k$ iterations (columns). Thick numbers indicate order jumps. }
	\label{table:MINRADAU}
\end{table}
\begin{table}[h]
	\bgroup
	\def\arraystretch{1}%
	\begin{tabular}{c||c|c|c|c|c|c|c|c|c|c|c|c|c|c|c|}
		 $s$ $\backslash$ $k$ & 1 & 2 & 3 & 4 & 5 & 6 & 7 & 8 & 9 & 10 & 11 & 12 & 13 & 14 & 15 \\
		\hline
		\hline
		1 & $ 1 $ & 1 & 1 & 1 & 1 & 1 & 1 & 1 & 1 & 1 & 1 & 1 & 1 & 1 & 1 \\
		\hline
		2 & $\bm{2}$ & 3 & 3 & 3 & 3 & 3 & 3 & 3 & 3 & 3 & 3 & 3 & 3 & 3 & 3 \\
		\hline
		3 & $\bm{2}$ & $\bm{4}$ &  5 & 5 & 5 & 5 & 5 & 5 & 5 & 5 & 5 & 5 & 5 & 5 & 5\\
		\hline
		4 & $\bm{2}$ & $\bm{4}$ & $\bm{6}$ & 7 & 7 & 7 & 7 & 7 & 7 & 7 & 7 & 7 & 7 & 7 & 7 \\
		\hline
		5 & $\bm{2}$ & $\bm{4}$ & $\bm{6}$ & $\bm{8}$ & 9 & 9& 9 & 9 & 9 & 9 & 9 & 9 & 9 & 9 & 9 \\
		\hline
		6 & $\bm{2}$ & $\bm{4}$& $\bm{6}$ & $\bm{8}$ & $\bm{10}$ &11&11&11&11&11&11&11&11&11&11 \\
		\hline
		7 & $\bm{2}$ & $\bm{4}$ & $\bm{6}$ & $\bm{8}$ & $\bm{10}$ & $\bm{12}$ & 13 & 13 &13&13&13&13&13&13&13 \\
		\hline
		8 & $\bm{2}$ & $\bm{4}$ & $\bm{6}$ & $\bm{8}$ & $\bm{10}$ & $\bm{12}$ & $\bm{14}$ & 15 &15&15&15&15&15&15&15 \\
		\hline
	\end{tabular}
	\egroup
	\caption{Convergence order of SDC using the order-jump EED characterized in this paper that ensures two orders per iteration. The method uses $s$ Radau IIA nodes (lines) for a maximum order of $2s -1$ and $k$ iterations (columns). Thick numbers indicate order jumps.}
	\label{table:JUMPERRADAU}
\end{table}
\paragraph{Structure of the paper.}
This work is organised as follows. In chapter 2 we derive SDC based on a perturbed initial value problem and bridge a connection to RKM. We introduce Butcher series and their connection to SDC. Chapter 3 contains the main theoretical results of this work. Using Butcher series, we prove that an SDC method with $K$ iterations has at least order $K$. We proceed by using linear theory to find an EED that increases the order of SDC by two in each iteration, independent of the node set. We then derive conditions on EED that increases the order of SDC by two for nonlinear problems if a collocation \reva{or discontinuous} collocation method is used as the underlying method. We construct a diagonal order-jump EED for parallel SDC methods which performs an order jump on each iteration. In Chapter 4 we discuss linear stability properties of the new methods and possible improvements of the stability domain. We numerically illustrate that the convergence order predicted in our theory is achieved. We then discuss convergence of the iterations of the SDC methods to the underlying method with a given timestep.

\section{Spectral deferred corrections}\label{sec:sdc}

Consider the system of ordinary differential equations (ODE),
\begin{equation}
    \frac{\diff u(t)}{\diff t} = f(u(t)),
	\label{IVP}
\end{equation}
with vector field\footnote{For simplicity of presentation, we assume that the ODE is autonomous. Note that this is without loss of generality, as a non-autonomous system can be made autonomous by considering the augmented system with time equation $\frac{ds(t)}{dt} = 1$, see \citep[Section III.1.1]{HairerEtAl2002_geometric}.} $f: \mathbb{R}^d \to \mathbb{R}^d$ assumed sufficiently differentiable and with initial condition $u(0) = u_0$.
Consider a classical implicit Runge-Kutta methods (RKM) of the form
\begin{subequations}
\begin{align}
		u_{n+1} &= u_n + \Delta t \sum_{i=1}^s b_i f(u_n^{[i]} )\label{ButcherUpdate} \,, \\
		u_n^{[i]} &=u_n + \Delta t \sum_{j=1}^s a_{ij} f(u_n^{[j]} ) \,, \quad i = 1, \dots, s \,,
	\label{ButcherForm}
\end{align}
\end{subequations}
where $\Delta t = t_{n+1}- t_n$ is the time step and the coefficients $b_i$ and $a_{ij}$  define the RKM and we define $c_i := \sum_{j=1}^s a_{ij}$. 
The coefficients of the RKM can be compactly arranged in a Butcher tableau
\begin{equation*}
\renewcommand\arraystretch{1.2}
\begin{array}
{c|c}
\bm{c} & \bm{A}\\
\hline
&\bm{b}^\intercal
\end{array},
\label{ButcherTableau}
\end{equation*}
where $\bm{c} = (c_i)_{i=1}^s$, $\bm{b} = (b_i)_{i=1}^s$, and $\bm{A} = (a_{ij})_{i,j=1}^s$ \citep{Butcher1964a}. In this paper, we focus on \emph{collocation} Runge-Kutta methods \citep[Section II.1.2]{HairerEtAl2002_geometric} with arbitrary collocation nodes $c_1, \ldots, c_s$ where the Runge-Kutta coefficients are given by
\[ b_i = \int_{0}^1 l_i(x) dx \,, \quad a_{ij} = \int_{0}^{c_i} l_j(x) dx \,, \quad i, j = 1, \dots, s \,, \]
where $l_i(x) = \prod_{j\neq i} \frac{x - c_j}{c_i - c_j}$ are the Lagrange interpolation polynomials related to the nodes $c_i, i=1, \dots, s$. It includes, in particular, the widely used Gauss, Radau IIA, Lobatto IIIA methods which are implicit Runge-Kutta methods of order $2s$, $2s - 1$, and $2s - 2$. \reva{Our study also includes the variant of \emph{discontinuous collocation} methods \citep[Section II.1.2]{HairerEtAl2002_geometric} which includes Radau IA, Lobatto IIIB, and Lobatto IIIC.}

Consider the Dahlquist's test equation $\frac{\diff u(t)}{\diff t} = \lambda u(t)$ with $\lambda \in \mathbb{C}$ and $Re(\lambda) \leq 0$. 
For any method, we can find the stability function $R(z)$, with $z = \lambda \Delta t$, that advances the numerical solution of the Dahlquist equation forward in time $u_{n+1} = R(z)u_n$.
Following \citep{HairerEtAl1996_stiff} (Definition 2.1) we call the set $S = \{z \in \mathbb{C}: |R(z)| \leq 1 \}$ the stability domain of a given method. 
If $z \in S$, then the numerical solution $(u_n)_{n \in \mathbb{N}}$ remains bounded for $n \to \infty$. 
\revb{A method is called A-stable if $S \subset \mathbb{C}^{-}$.}
A method is called L-stable if it is A-stable and additionally satisfies \citep{Ehle1969}
\begin{equation}
	\lim_{z \to \infty} R(z) = 0 .
\label{Lstable}
\end{equation}
A method that satisfies $a_{sj} = b_j$, $j = 1, \dots, s$ is L-stable~\citep[Prop. 3.8]{HairerEtAl1996_stiff}. 
These methods are called stiffly accurate \citep{Prothero1974}.
For instance, Radau IIA and Lobatto IIIA are L-stable and stiffly accurate RKM.

An alternative to using the classical Newton or quasi-Newton method for solving the non-linear system \eqref{ButcherForm} and compute the internal stages $u^{[i]}_n$ is to consider SDC methods. 
As emphasized in the introduction, there are many possible variants of SDC methods. 
Here we consider the following derivation based on the integral formulation
\begin{equation}
    \label{eq:SDC_intro1}
	u(t) = u(t_n) + \int_{t_n}^t f(u(s)) \diff s = u(t_n) + \int_{t_n}^t {f}(\tilde{u}(s)) \diff s + \int_{t_n}^t ({f}(u(s))-{f}(\tilde{u}(s)))\diff s
\end{equation}
of~\eqref{IVP}, where $\tilde{u}(s)$ stands for the approximation of the solution $u(s)$.
Performing a Picard iteration, we obtain the following iteration $u^{k+1}(t)$ to approximate $u(t)$,
\begin{equation}
    u^{k+1}(t) = u_n + \int_{t_n}^t {f}(u^{k}(s)) \diff s + \int_{t_n}^t (f(u^{k+1}(s))-f(u^{k}(s)))\diff s \,, \quad k = 0, 1, 2, \dots \,,
\end{equation}
where $u_n$ is a known numerical solution at time $t_n$.
Applying Runge-Kutta discretizations to each of the two above integrals results in the following SDC method,
\begin{equation}
    \label{eq:SDC_internal}
    u_n^{[k+1,i]} = u_n + \underbrace{\Delta{t}\sum_{j=1}^{s} a_{ij}f(u_n^{[k,j]})}_{\text{collocation term}} + \underbrace{\Delta{t}\sum_{j=1}^{s} \tilde{a}^k_{ij} (f(u_n^{[k+1,j]}) - f(u_n^{[k,j]}))}_{\text{correction term}} \,,
\end{equation}
where $k = 1, \dots, K$ for a fixed number of iterations $K$ and where $u_n^{[K,i]}$ approximates $u_n^{[i]}$. The Runge-Kutta coefficients $\bm{A}_\Delta^k = (\tilde{a}_{ij}^k)_{i,j=1}^s$ correspond to the so called \emph{error equation discretization} (EED) where the corresponding term in \eqref{eq:SDC_internal} can be interpreted as a correction term in the fixed point Picard iterations. The output of the SDC method is then given by
\begin{equation}
    \label{eq:SDC_final}
    u_{n+1} = u_n + \Delta t \sum_{i=1}^s b_i f(u^{[K,i]}_n) \,.
\end{equation}
The iterative method \eqref{eq:SDC_internal} and \eqref{eq:SDC_final} defines the class of SDC methods $(\bm{A}^0_\Delta, \dots, \bm{A}^K_\Delta, \bm{A}, \bm{b})$ with $\bm{A}^k_\Delta = (\tilde{a}^k_{ij})_{i,j=1}^s$, that we shall analyse in this paper.

\begin{remark}
    \revc{Equation \eqref{eq:SDC_intro1} can be interpreted as $u(t) = \tilde{u}(t) + \tilde{\delta}(t)$ where $\tilde{u}(t)$ and $\tilde{\delta}(t)$ solve the following partitioned problem,
    \begin{equation}
    \frac{\diff\tilde{u}(t)}{\diff t} = f(t, \tilde{u} + \tilde{\delta}) \,, \quad \quad \frac{\diff\tilde{\delta}(t)}{\diff t} = f(t, \tilde{u} + \tilde{\delta}) - f(t, \tilde{u}) \,,
    \end{equation}
where $ \tilde{u}(t)$ corresponds to the SDC approximation of $u(t)$ and $\tilde{\delta}(t)$ is the error of the approximate solution $\tilde{u}$ and we assume that $\tilde{\delta}(t_n) = 0$.}
\end{remark}

The SDC method~\eqref{eq:SDC_internal} and~\eqref{eq:SDC_final} is itself a Runge-Kutta method with the following augmented Butcher tableau with $(K+1)s$ internal stages,
\begin{equation}
        \label{eq:SDCButcherTableau}
\begin{array}
{c|ccccc}
\tilde{\bm{c}}& \bm{A}_{\Delta}^0  &  &     &  \\
\bm{c}& \bm{A}-\bm{A}_{\Delta}^1 & \bm{A}_{\Delta}^1 &   &  \\
\vdots&     & \ddots & \ddots &   \\
\bm{c} & &    & \bm{A}-\bm{A}_{\Delta}^K & \bm{A}_{\Delta}^K \\
\hline
&  &  &   & \bm{b}^\intercal 
\end{array} \,,
\end{equation}
where $\bm{A}$ is the coefficients of the original Runge-Kutta method \eqref{ButcherForm} and $\tilde{\bm{c}} = (\tilde{c}_{i})_{i=1}^s$ with $\tilde{c}_i = \sum_{j=1}^s \tilde{a}^0_{ij}$.
Natural choices for the EED $\bm{A}^k_\Delta$ are based on implicit or explicit Euler or diagonal to enable parallelization
\begin{equation}
\label{eq:IEEulerED}
\bm{A}_\Delta^k= \begin{pmatrix}
\Delta\tau_1 \\
\vdots & \ddots & \\
\Delta\tau_1 & \dots & \Delta\tau_{s} 
\end{pmatrix}\,,
\quad
\bm{A}_\Delta^k = \begin{pmatrix}
0 & \\
\Delta\tau_2 &  0\\
\vdots & \ddots & \ddots\\
\Delta\tau_2 & \dots & \Delta\tau_{s} & 0
\end{pmatrix} \,,
\quad
\bm{A}_\Delta^k = \alpha_k \diag(\bm{c}) \,,
\end{equation} 
with $\Delta \tau_i = c_{i} - c_{i-1}$ and $\Delta \tau_1  = c_{1}$, for some constants $\alpha_k \in \mathbb{R}$.

\begin{remark}
    \revc{Let $(\bm{A}_\Delta^0, \ldots, \bm{A}^K_\Delta, \bm{A}, \bm{b})$ define an implicit SDC method. If the Butcher tableau \eqref{eq:SDCButcherTableau} is not singular and satisfies the stiffly accurate condition $a_{si} = b_i$ for $i=1,\dots,s$, then \eqref{Lstable} is satisfied. This is an immediate consequence of the classical result \cite[Prop. 3.8]{HairerEtAl1996_stiff} for standard Runge--Kutta methods, see also \citep{VanDerHouwen1991} in the context of SDC.
    Preserving the A-stability or L-stability of the underlying Runge--Kutta method $(\bm{A}, \bm{b})$ is however not straightforward when constructing SDC methods, as discussed in Section \ref{sec:numexp}.}
\end{remark}

\begin{remark}
    \label{rmk:init}
    For the sake of simplicity, we initialize the SDC method~\eqref{eq:SDC_internal},~\eqref{eq:SDC_final} by copying the value brought forward from the previous timestep, so that $u^{[0,i]}_n = u_n$, $i = 1, \dots, s$. 
    However, other choices are possible~\citep{Ascher1997}, for example
    \begin{equation*}
        u_n^{[0,i]} = u_n + \Delta t \sum_{j=1}^{s} \tilde{q}_{ij} f(u_n^{[0,j]} )
    \end{equation*}
    for a choice of coefficients $\tilde{q}_{ij}$ that yields a higher order of approximation of $u_{n+1}$ and satisfies $\bm{c} = \tilde{\bm{c}}$ to approximate the internal stages. 
    This initialization corresponds to setting $\bm{A}^0_\Delta = (\tilde{q}_{ij})_{i,j=1}^s$ in our formulation of SDC: using~\eqref{eq:SDC_internal}, we find that
    \begin{equation*}
        u_n^{[1, i]} = u_n + \Delta{t}\sum_{j=1}^{s} a_{ij} f(u_n) + \Delta{t}\sum_{j=1}^{s} \tilde{q}_{ij} (f(u_n^{[1,j]}) - f(u_n)) = u_n + \Delta t \sum_{j=1}^s \tilde{q}_{ij} f(u_n^{[1,j]}) \,,
    \end{equation*}
    where $\sum_{j=1}^{s} a_{ij} = c_i = \tilde{c}_i = \sum_{j=1}^{s} \tilde{q}_{ij}$.
\end{remark}

\begin{remark}
    We emphasize that the assumption $\sum_{j=1}^s \tilde{a}^k_{ij} = c_i$ used in Remark~\ref{rmk:init} for $k = 0$ is not needed to define the EED $\bm{A}^k_\Delta$ with $k \geq 1$. 
    For instance, the second and third choices in~\eqref{eq:IEEulerED} do not satisfy this assumption that the EED is consistent with the nodes $\bm{c}$ of the original Runge-Kutta method.
\end{remark}

\begin{remark}
    \label{rmk:final_update}
    \reva{There are different possible choices for the weights $\bm{b}$ in the linear combination of the stages~\eqref{eq:SDC_final} that computes the approximation $u_{n+1}$ at the end of the timestep. 
    For instance, following \citep[Eq. 4.2]{DuttEtAl2000}, one can consider the extrapolation formula $u_{n+1} = \sum_{i=1}^s \ell_i(1) u_n^{[K,i]}$, where $\ell_i$ are the Lagrange polynomials to the nodes $c_i$.
    For stiffly accurate Runge-Kutta methods,  like Radau IIA or Lobatto IIIA methods, where $a_{sj} = b_j$, setting $u_{n+1} = u_n^{[K,s]}$ can produce L-stable SDC methods~\citep[Theorem 3.1]{LaytonMinion2005}.
    Using the coefficients $b_i$ of the underlying RKM can be advantageous compared to $u_{n+1} = u^{[K,s]}_n$, because it adds one order of accuracy without extra cost. 
    The outlined choices can all be realized by choosing different coefficients $\bm{b}$~\citep{Fregin2026}. while the coefficients $\bm{A}, \bm{A}_\Delta^k, \bm{c}$ remain the same.}
\end{remark}

\subsection{Butcher series for SDC}  \label{sec:Bseries}
Butcher series are a powerful tool for the numerical analysis of Runge-Kutta methods and hence are a natural choice to study of SDC.
Expanding in Taylor series as $h \to 0$ the exact solution and numerical solution after one step and assuming $u(t_n) = u_n$, we obtain
\begin{align}
	u_{n+1} &= u_n + \Delta t \sum_{i=1}^s b_i f(u_n) + \Delta t^2 \sum_{i,j=1}^s b_i a_{ij} f^\prime (u_n) f (u_n) + \mathcal{O}(\Delta t^{3}), \nonumber \\
    u(t_{n+1}) &= u_n + \Delta t f(u_n) + \Delta t^2 \frac{1}{2} f^\prime (u_n) f (u_n) + \mathcal{O}(\Delta t^{3}).
\label{eq:Taylor_expansion_example}
\end{align}
The difference $u_{n+1} - u(t_{n+1}) = \mathcal{O}(h^{p+1})$ corresponds to the \emph{local error} of an order $p$ method. In particular, we see from \eqref{eq:Taylor_expansion_example} that Runge-Kutta method with coefficients $(\bm{A},\bm{b})$ is of order (at least) $1$ if and only if $\sum_{i=1}^s b_i = 1$ and of order (at least) $2$ if and only if
\[ \sum_{i=1}^s b_i = 1, \quad \sum_{i,j=1}^s b_i a_{ij} = \frac{1}{2}. \]
Rooted trees~\citep{Cayley1857} systematically represent the elementary differentials which arise in the Taylor expansion of the exact solution and the Runge-Kutta method. The advantage of Butcher series is that they are able to represent both the exact and numerical solution~\citep{hairerButcherGroupGeneral1974} based on the seminal work of J. Butcher \citep{Butcher1964a}. 
We follow the presentation in~\citep[Section III]{HairerEtAl2002_geometric}.
\begin{definition}[rooted tree]
A \emph{rooted tree} is defined recursively as a vertex called the \emph{root} and an unordered monomial of rooted trees called the \emph{children} of the root.
\end{definition}
The set of all rooted trees is denoted by $T$.
Examples of trees with roots at the bottom are
\[ \forestA, \quad \forestB, \quad \forestC, \quad \forestD. \]
Monomials of trees are called \emph{forests} and denoted by $S(T)$. 
The empty monomial is written as $\mathbf{1}$. 
\reva{Let $\{ \tau_1 \cdots \tau_n \}$ denote the tree with children $\tau_1, \dots, \tau_n \in T$.
For example,}
\[ \reva{\{\mathbf{1}\}} = \forestE, \quad \reva{\{\forestF\}} = \forestG, \quad \reva{\{\forestH\}} = \forestI, \quad \reva{\{\forestJ\}} = \forestK. \]
We also define several functions over $T$ that are useful for a systematic representation of order conditions.
\begin{definition}[tree functions]
    \label{def:tree_functions}
    The size, factorial, symmetry, and elementary differential of a tree $\tau = \reva{\{ \tau_1^{m_1} \cdots \tau_k^{m_k} \}}$ with distinct trees $\tau_1, \dots, \tau_k \in T$ with multiplicities $m_1, \dots, m_k \in \mathbb{N}$ are defined as follows:
	\begin{enumerate}
		\item the \emph{size} is defined as the number of vertices in $\tau$, that is, $|\tau| := 1 + \sum_{i=1}^k m_i |\tau_i|$,
		\item the \emph{factorial} is defined as the product of the size of the tree and the factorials of the children of the root, that is, $\tau ! = |\tau| \prod_{i=1}^k (\tau_i !)^{m_i}$ with $\forestL! = 1$,
		\item the \emph{symmetry} is defined as the number of automorphisms of the tree, that is, $\sigma(\tau) := \prod_{i=1}^k m_i! \sigma(\tau_i)^{m_i}$ with $\sigma(\forestM) = 1$.
		\item the \emph{elementary differential} $F_{\Delta t f} (\tau) : \mathbb{R}^d \to \mathbb{R}^d$ is defined as
        \[ F_{\Delta t f} (\tau) = \Delta t \sum_{i_1, \dots, i_N = 1}^d F_{\Delta t f} (\tau_1)^{i_1} \cdots F_{\Delta t f} (\tau_N)^{i_N} \frac{\partial f}{\partial x_{i_1} \cdots \partial x_{i_N}}. \]
      where $\tau = \reva{\{ \tau_1 \cdots \tau_N \}}$ with $\tau_1 \cdots \tau_N$ not distinct, and $F_{\Delta t f} (\forestN) := \Delta t f$.
	\end{enumerate}
\end{definition}
Several examples of the elementary differentials are
\[ F_{\Delta t f} (\forestO) = (\Delta t)^2 f^\prime f \,, \quad F_{\Delta t f} (\forestP) = (\Delta t)^3 f^\prime f^\prime f \,, \quad F_{\Delta t f} (\forestQ) = (\Delta t)^3 f^{\prime\prime} (f,f) \,, \]
\revb{where $f^\prime$ is the first derivative of $f$ (represented by the Jacobian) and $f^{\prime\prime}$ is the second derivative of $f$ (recall that $f^{\prime\prime}$ is a bilinear map $\mathbb{R}^d \times \mathbb{R}^d \to \mathbb{R}^d$).}
We are now able to represent the Taylor expansion in~\eqref{eq:Taylor_expansion_example} using formal sums indexed by rooted trees~\citep{Butcher1964a,HairerEtAl1993_nonstiff,Bseries_intro,SanzSerna15fsa},
\begin{align}
	u_{n+1} &= u_n + \sum_{\tau \in T} \frac{\alpha(\tau)}{\sigma(\tau)} F_{\Delta t f} (\tau) (u_n), \nonumber \\
	\label{eq:Taylor_expansion_tree}
	u(t_{n+1}) &= u_n + \sum_{\tau \in T} \frac{1}{\tau! \sigma(\tau)} F_{\Delta t f} (\tau) (u_n),
\end{align}
where $\alpha : T \to \mathbb{R}$ is a functional uniquely defined by the Butcher tableau of the Runge-Kutta method, for example
\[ \alpha(\bullet) = \sum_{i=1}^s b_i, \quad \alpha(\forestR) = \sum_{i,j=1}^s b_i a_{ij}, \quad \alpha(\forestS) = \sum_{i,j,k,l=1}^s b_i a_{ij} a_{ik} a_{kl}. \]
The formal sums in \eqref{eq:Taylor_expansion_tree} are called Butcher series. They give us explicit expressions for order conditions of Runge-Kutta methods of any order $p$. 
A Runge-Kutta method is of order $p$ if 
\begin{equation*}
\alpha(\tau) = \frac{1}{\tau!} \; \text{for all} \; |\tau| \leq p.
\end{equation*}
\begin{definition}[Butcher series]
    \label{def:Butcher_series}
    Let $\alpha : T \to \mathbb{R}$ be a coefficient map over trees, $\Delta t$ a timestep, and $f : \mathbb{R}^d \to \mathbb{R}^d$ a vector field. Then,
	\[ B_{\Delta t f}(\alpha) = \sum_{\tau \in T} \frac{\alpha(\tau)}{\sigma(\tau)} F_{\Delta t f} (\tau) \]
	is a formal sum of vector fields and called a \emph{Butcher series}.
\end{definition}
We note that both $u_{n+1}$ and $u_n^{[i]}$ of a Runge-Kutta method of the form \eqref{ButcherUpdate},\eqref{ButcherForm} can be written using Butcher series as
\begin{align*}
	u_{n+1} &= u_n + B_{\Delta t f}(\beta) (u_n) \\
    u^{[i]}_n &= u_n + B_{\Delta t f}(\beta^{[i]}) (u_n) \,, \quad i=1,\dots,s \,.
\end{align*}
The coefficient maps $\beta$ and $\beta^{[i]}$ are defined recursively by
\begin{align}
    \beta (\reva{\{ \tau_1 \cdots \tau_n \}}) &= \sum_{j=1}^s b_j \reva{\beta^{[j]}(\tau_1) \cdots \beta^{[j]}(\tau_n)}, \nonumber \\
    \beta^{[i]} (\reva{\{ \tau_1 \cdots \tau_n \}}) &= \sum_{j=1}^s a_{ij} \reva{\beta^{[j]}(\tau_1) \cdots \beta^{[j]}(\tau_n)}, \label{eq:ai_prop}
\end{align}
\reva{with $\beta(\bullet) = \sum_{j=1}^s b_j$ and $\beta^{[i]}(\bullet) = \sum_{j=1}^s a_{ij}$.}
In the context of SDC methods defined by \eqref{eq:SDC_internal}, \eqref{eq:SDC_final}, let $\alpha^{[K]}$ and $\alpha^{[K,i]}$ denote the final and internal stages of SDC $(\bm{A}^0_\Delta, \dots, \bm{A}^K_\Delta, \bm{A}, \bm{b})$ obtained after $K$ iterations, 
precisely, 
\begin{align*}
	u_n^{[K,i]} &= u_n + B_{\Delta t f} (\alpha^{[K,i]})(u_n), \\
	u_{n+1} &= u_n + \sum_{i=1}^s b_i f(u^{[K,i]}_n) = u_n + B_{\Delta t f} (\alpha^{[K]})(u_n).
\end{align*}
Combining \eqref{eq:ai_prop} with the form of the Butcher tableau of SDC in \eqref{eq:SDCButcherTableau}, we have the following property of $\alpha^{[K]}$ and \reva{$\alpha^{[k,i]}$ with $\bm{A}^k = (\tilde{a}^k_{ij})$ for $k=1,\dots,K$},
\begin{align}
    \alpha^{[K]}(\{\tau_1 \cdots \tau_n\}) &= \sum_{i=1}^s b_i \reva{\alpha^{[K,i]}(\tau_1) \cdots \alpha^{[K,i]}(\tau_n)}, \nonumber \\
    \alpha^{[k,i]} (\{\tau_1 \cdots \tau_n\}) &= \sum_{j=1}^s \big( a_{ij} - \tilde{a}^k_{ij} \big) \reva{\alpha^{[k-1,j]}(\tau_1) \cdots \alpha^{[k-1,j]}(\tau_n)} \nonumber \\ 
                                              &\quad \quad + \sum_{j=1}^s \tilde{a}^k_{ij} \reva{\alpha^{[k,j]} (\tau_1) \cdots \alpha^{[k,j]} (\tau_n)},
    \label{eq:ai_prop_sdc} ,
\end{align}
\reva{as well as $\alpha^{[0,i]}(\tau) = 0$ for all $\tau \in T$, $\alpha^{[K]} (\bullet) = \beta(\bullet)$, and $\alpha^{[k,i]} (\bullet) = \beta^{[i]} (\bullet)$.}
This property is used extensively in the proofs that follow.

\section{Convergence analysis of SDC methods}\label{sec:results}

Consider an SDC method $(\bm{A}_{\Delta}^0, \ldots, \bm{A}_{\Delta}^K, \bm{A}, \bm{b})$.
Since SDC approximates the internal stages of the Runge-Kutta method $(\bm{A}, \bm{b})$, it is not meant to achieve a higher order of accuracy for the exact solution than the RKM.
\reva{Therefore, we introduce the term \emph{iteration order} of an SDC method $(\bm{A}_{\Delta}^0, \ldots, \bm{A}_{\Delta}^K, \bm{A}, \bm{b})$ as being determined} by the local error of final stage \eqref{eq:SDC_final} compared to the Runge-Kutta method $(\bm{A}, \bm{b})$ which motivates the following definition.
\begin{definition}[SDC order]
    \label{def:sdc_order}
    An SDC method $(\bm{A}_{\Delta}^0, \ldots, \bm{A}_{\Delta}^K, \bm{A}, \bm{b})$ is of \emph{\reva{iteration} order} $p_K$ if it coincides with the Runge-Kutta method $(\bm{A}, \bm{b})$ after one step, up to terms of order $\mathcal{O}(\Delta t^{p_K+1})$, i.e.,
    \[ \alpha^{[K]}(\tau) = \beta (\tau) \,, \quad \text{for } |\tau| \leq p_K \,. \]
    \revb{In addition, internal stages of an SDC method are of iteration order $p_K$ if
    \begin{equation} 
    \alpha^{[K,i]}(\tau) = \beta^{[i]} (\tau) \,, \quad \text{for } i = 1, \dots, s \quad \text{and } |\tau| \leq p_K \,.
    \end{equation} }
\end{definition}

The \reva{iteration} order should not be confused with the order of the Runge-Kutta method \eqref{ButcherUpdate}, \eqref{ButcherForm} or the Runge-Kutta method with the Butcher tableau~\eqref{eq:SDCButcherTableau} which, by contrast, is based on the local error compared to the exact solution,
\[ \alpha^{[K,i]}(\tau) = \frac1{\tau!} \,, \quad \text{for } i = 1, \dots, s \quad \text{and } |\tau| \leq p_K \,. \]
Now consider the height of a tree and the corresponding concept of height order of SDC.
\begin{definition}[tree height]
    \label{def:tree_height}
The \emph{height} of a tree $\tau = \reva{\{\tau_1 \cdots \tau_n\}}$ with trees $\tau_1, \dots, \tau_n \in T$ is defined as the number of vertices in the longest path from the root to a leaf, that is, $\text{ht}(\tau) := 1 + \max_{i=1}^n \text{ht}(\tau_i)$ with $ht(\forestT) := 1$.
\end{definition}
We note that for any tree $\tau \in T$, the size of the tree $\tau$ is greater than its height, that is, $|\tau| \geq ht(\tau)$, see Table~\ref{table:HeightAndSize}.
\renewcommand{\figurename}{Table}
\begin{figure}
	\centering
	\bgroup
	\def\arraystretch{3}%
	\setlength\tabcolsep{0.5cm}
	\begin{tabular}{c||c|c|c|c|}
		& size 2 & size 3 & size 4 & size 5 \\
		\hline
		\hline
		height 2 & \forestU & \forestV & \forestW & \forestX \\
		\hline
		height 3 & & \forestY & \forestAB, \forestBB & \forestCB, \forestDB, \forestEB, \forestFB \\
		\hline
		height 4 & & & \forestGB & \forestHB, \forestIB, \forestJB  \\
		\hline
		height 5 & & & & \forestKB \\
		\hline
	\end{tabular}
	\egroup
	\caption{Comparison of the size and height of trees.}
	\label{table:HeightAndSize}
\end{figure} 
\renewcommand{\figurename}{Fig.}
\begin{definition}[height order]
    An SDC method $(\bm{A}_{\Delta}^0, \ldots, \bm{A}_{\Delta}^K, \bm{A}, \bm{b}, \bm{c})$  has \emph{height order} $h_K$ if
    \[ \alpha^{[K]} (\tau) = \beta (\tau), \quad \text{for } i = 1, \dots, s, \quad ht(\tau) \leq h_K \,. \]
    \revb{In addition, internal stages of an SDC method are of height order $h_K$ if
    \begin{equation} 
    \alpha^{[K,i]}(\tau) = \beta^{[i]} (\tau) \,, \quad \text{for } i = 1, \dots, s \quad \text{and } ht(\tau) \leq h_K \,.
\end{equation} }

\end{definition}

Since $|\tau| \geq ht(\tau)$ for any given tree $\tau$, meaning its size is greater than its height, the set of all trees of height up to $h_k$ includes the set of all trees up to size $h_k$.
Therefore, an SDC method with height order $h_k$ is (at least) of \reva{iteration} order $h_k$. 
We also note that there is an infinite number of trees of a given height $ht(\tau) \geq 2$, so that the B-series of an SDC method with height order of at least $2$ coincides with the B-series of the Runge-Kutta method $(\bm{A}, \bm{b})$ for an infinite number of trees.

\subsection{Convergence order of SDC}
\label{sec:convergence_order_sdc}

The following theorem shows that an SDC iteration with an arbitrary EED increases the height order of SDC by one.
\begin{theorem}\label{thm:order_per_sweep}
    Let the internal stages of an SDC method $(\bm{A}_{\Delta}^0, \ldots, \bm{A}_{\Delta}^K, \bm{A}, \bm{b})$ be of height order $h_K - 1$.
    It then holds that
    
    \begin{enumerate}
        \item the method has height order $h_K$,
        \item the internal stages of $(\bm{A}_{\Delta}^0, \ldots, \bm{A}_{\Delta}^K, \bm{A}_{\Delta}^{K+1}, \bm{A}, \bm{b}, \bm{c})$ are of height order $h_K$ for any $\bm{A}^{K+1}_\Delta$.
    \end{enumerate}
\end{theorem}
\begin{proof}
    \reva{Let us start by proving the first statement. We show that $\alpha^{[K]}(\tau) = \beta(\tau)$ for all trees $\tau$ of height up to $h_K$.
    Let $\tau = \{ \tau_1 \cdots \tau_n \}$ be a tree of height up to $h_K$, $ht(\tau) \leq h_K$, with $\tau_i$ being trees of height up to $h_K - 1$, then, use~\eqref{eq:ai_prop_sdc},}
    \begin{align*}
        \alpha^{[K]}(\tau) &= \sum_{i=1}^s b_i \reva{\alpha^{[K,i]} (\tau_1) \cdots \alpha^{[K,i]} (\tau_n)} \\
        \intertext{\reva{apply the hypothesis $\alpha^{[K,i]} (\tau_j) = \beta^{[i]} (\tau_j)$ for $j = 1, \dots, n$,}}
                           &= \sum_{i=1}^s b_i \reva{\beta^{[i]} (\tau_1) \cdots \beta^{[i]} (\tau_n)} = \beta(\tau),
    \end{align*}
    \reva{and the first statement is proved. For the second statement, we apply induction. We use~\eqref{eq:ai_prop_sdc} again and note that $\alpha^{[K+1,i]} (\bullet) = \beta^{[i]}(\bullet)$. Let $\tau = \{ \tau_1 \cdots \tau_n \}$ be a tree of height up to $h_K$, $ht(\tau) \leq h_K$, with $\tau_i$ being trees of height up to $h_K - 1$, then,}
    \begin{align*}
        \alpha^{[K+1,i]} (\tau) &= \sum_{j=1}^s (a_{ij} - \tilde{a}^{K+1}_{ij}) \reva{\alpha^{[K,j]}(\tau_1) \cdots \alpha^{[K,j]}(\tau_n)} \\
                                &\quad\quad\quad\quad + \sum_{j=1}^s \tilde{a}^{K+1}_{ij} \reva{\alpha^{[K+1,j]}(\tau_1) \cdots \alpha^{[K+1,j]}(\tau_n)}
                                \intertext{\reva{apply the hypothesis $\alpha^{[K,j]}(\tau_j) = \beta^{[i]} (\tau_j)$ and induction hypothesis $\alpha^{[K+1,j]}(\tau_j) = \beta^{[i]} (\tau_j)$ for $j = 1, \dots, n$,}}
        &= \reva{\sum_{j=1}^s a_{ij}\beta^{[j]}(\tau_1) \cdots \beta^{[j]}(\tau_n)} = \beta^{[i]} (\tau),
    \end{align*}
    and the second statement is proved.
\end{proof}

Note that Theorem \ref{thm:order_per_sweep} remains valid if we replace height order $h_K$ by \reva{iteration} order $p_K$. 
\begin{corollary}
    \label{corr:hk=K}
    Consider an SDC method $(\bm{A}_\Delta^0, \bm{A}_{\Delta}^1, \ldots, \bm{A}_{\Delta}^K, \bm{A}, \bm{b})$ with $K$ iterations, then it has height order $h_K \geq K$ and \reva{iteration} order $p_K \geq K$.
\end{corollary}

\begin{remark}
\reva{An immediate consequence of Corollary \ref{corr:hk=K} is that we obtain the SDC iteration order $p_K \geq h_K \geq K$ for the SDC method $(\bm{A}_\Delta^0, \bm{A}_{\Delta}^1, \ldots, \bm{A}_{\Delta}^K, \bm{A}, \bm{b})$ if we use any of the EEDs introduced in~\citep{CaklovicEtAl2025}.
Choosing
    \begin{equation}
        \label{eq:Caklovic_EEDs}
        \bm{A}_{\Delta}^k = \diag(\frac{\bm{c}}{s}) \,, \quad \bm{A}_{\Delta}^k = \diag(\frac{\bm{c}}{k}) \quad \text{or} \quad \bm{A}_{\Delta}^k = \argmin_{\bm{A}_\Delta^k = \diag(\hat{\bm{c}})} \lambda_{\max}(\bm{K}^k_S) \,,
    \end{equation}
    generates their MIN-SR-NS, MIN-SR-FLEX and MIN-SR-S variant, respectively.}
\end{remark}

\begin{remark}
    \label{rmk:trap_order_jump_Lobatto}
    If an EED based on the trapezoidal method is used by averaging the first two EEDs in~\eqref{eq:IEEulerED}, the iteration directly after an order jump will fail to increase the order of the SDC method. This can be explained by the fact that the order jumps do not increase the height order of all internal stages. Therefore, this phenomenon does not contradict Theorem~\ref{thm:order_per_sweep}.
\end{remark}

Let us illustrate Theorem \ref{thm:order_per_sweep} with an example of an SDC method with $K=2$ iterations based on Runge-Kutta method with $s=2$ stages.
\begin{example}
    Consider an SDC method $(\bm{A}_\Delta^0, \bm{A}_\Delta^1, \bm{A}_\Delta^2, \bm{A}, \bm{b})$ with $K=2$ with lower triangular EEDs where $( \bm{A}, \bm{b})$ is a $2$-stage Runge-Kutta method. The Butcher tableau \eqref{eq:SDCButcherTableau} of the SDC method has the form
    \[
        \renewcommand\arraystretch{1.5}
        \begin{array}{c|cccccc}
            c_1 & \tilde{a}^{0}_{11} & & & & & \\
            c_2 & \tilde{a}^{0}_{21} & a^{0\Delta}_{22} & & & & \\
            c_1 & a_{11} - \tilde{a}^{1}_{11} & a_{12} & \tilde{a}^{1}_{11} & & & \\
            c_2 & a_{21} - \tilde{a}^{1}_{21} & a_{22} - \tilde{a}^{1}_{22} & \tilde{a}^{1}_{21} & \tilde{a}^{1}_{22} & & \\
            c_1 & & & a_{11} - \tilde{a}^{2}_{11} & a_{12} & \tilde{a}^{2}_{11} & \\
            c_2 & & & a_{21} - \tilde{a}^{2}_{21} & a_{22} - \tilde{a}^{2}_{22} & \tilde{a}^{2}_{21} & \tilde{a}^{2}_{22} \\
            \hline
            & & & & & b_1 & b_2
        \end{array}
    \]

    Expanding the SDC method in B-series,
    \[ u_{n+1} = u_n + B_{hf}(\alpha^{[2]})(u_n), \]
    and comparing with the B-series of the Runge-Kutta method $(\bm{A}, \bm{b})$, we obtain,
    \begin{align*}
        \alpha^{[2]}(\bullet) &= b_1 + b_2 = \beta(\bullet), \\
        \alpha^{[2]}(\forestLB) &= b_1 (a_{11} - \tilde{a}^{2}_{11} + a_{12} + \tilde{a}^{2}_{11}) + b_2 (a_{21} - \tilde{a}^{2}_{21} + a_{22} - \tilde{a}^{2}_{22} + \tilde{a}^{2}_{21} + \tilde{a}^{2}_{22}) \\
        &= b_1 (a_{11} + a_{12}) + b_2 (a_{21} + a_{22}) = \beta(\forestMB), \\
        \alpha^{[2]}(\forestNB) &= b_1 (a_{11} - \tilde{a}^{2}_{11} + a_{12} + \tilde{a}^{2}_{11})^2 + b_2 (a_{21} - \tilde{a}^{2}_{21} + a_{22} - \tilde{a}^{2}_{22} + \tilde{a}^{2}_{21} + \tilde{a}^{2}_{22})^2 \\
        &= b_1 (a_{11} + a_{12})^2 + b_2 (a_{21} + a_{22})^2 = \beta(\forestOB), \\
        \alpha^{[2]}(\forestPB) &= b_1 \big( (a_{11} - \tilde{a}^{2}_{11}) (a_{11} - \tilde{a}^{1}_{11} + a_{12} + \tilde{a}^{1}_{11}) \\
        &\quad + a_{12} (a_{21} - \tilde{a}^{1}_{21} + a_{22} - \tilde{a}^{1}_{22} + \tilde{a}^{1}_{21} + \tilde{a}^{1}_{22}) \\
        &\quad + \tilde{a}^{2}_{11} (a_{11} - \tilde{a}^{2}_{11} + a_{12} + \tilde{a}^{2}_{11}) \big) \\
        &\quad + b_2 \big( (a_{21} - \tilde{a}^{2}_{21}) (a_{11} - \tilde{a}^{1}_{11} + a_{12} + \tilde{a}^{1}_{11}) \\
        &\quad + (a_{22} - \tilde{a}^{2}_{22}) (a_{21} - \tilde{a}^{1}_{21} + a_{22} - \tilde{a}^{1}_{22} + \tilde{a}^{1}_{21} + \tilde{a}^{1}_{22}) \\
        &\quad + \tilde{a}^{2}_{21} (a_{11} - \tilde{a}^{2}_{11} + a_{12} + \tilde{a}^{2}_{11}) \big) \\
        &\quad + \tilde{a}^{2}_{22} (a_{21} - \tilde{a}^{2}_{21} + a_{22} - \tilde{a}^{2}_{22} + \tilde{a}^{2}_{21} + \tilde{a}^{2}_{22}) \big) \\
        &= b_1 \big( (a_{11} - \tilde{a}^{2}_{11}) (a_{11} + a_{12}) + a_{12} (a_{21} + a_{22}) + \tilde{a}^{2}_{11} (a_{11} + a_{12}) \big) \\
        &\quad + b_2 \big( (a_{21} - \tilde{a}^{2}_{21}) (a_{11} + a_{12}) + (a_{22} - \tilde{a}^{2}_{22}) (a_{21} + a_{22}) + \tilde{a}^{2}_{21} (a_{11} + a_{12}) \big) + \tilde{a}^{2}_{22} (a_{21} + a_{22}) \big) \\
        &= b_1 \big( a_{11} (a_{11} + a_{12}) + a_{12} (a_{21} + a_{22}) \big) + b_2 \big(a_{21} (a_{11} + a_{12}) + a_{22} (a_{21} + a_{22}) \big) = \beta(\forestQB).
    \end{align*}
    We see that the equalities are true independent of the entries $\tilde{a}_{ij}^k$ of the EEDs $\bm{A}_\Delta^0, \bm{A}_\Delta^1, \bm{A}_\Delta^2$.
\end{example}

\subsection{Analysis of order jumps}
\label{sec:order_jumps}
In the following, we analyse the phenomenon when the order of accuracy with which the value $u_{n+1}$ is approximated increases by more than one in a single iteration. Such a phenonmenon is called an \emph{order jump}.

Recall that the Runge-Kutta coefficients of collocation methods of order $p$ using Gauss, Radau or Lobatto nodes satisfy the following simplifying assumptions~\citep{HairerEtAl1993_nonstiff,HairerEtAl1996_stiff}, see Table \ref{fig:simplifying_assumptions}, called $B(p), C(\eta), D(\zeta)$: 
\begin{equation}
    \label{eq:simplifying_assumptions}
    \begin{aligned}
        B(p) &:&\quad \sum_{i=1}^s b_i c_i^{q-1} &= \frac{1}{q}, \quad \text{for } q = 1, \dots, p; \\
        C(\eta) &:&\quad \sum_{j=1}^s a_{ij} c_j^{q-1} &= \frac{c_i^q}{q}, \quad \text{for } i = 1, \dots, s, \quad q = 1, \dots, \eta; \\
        D(\zeta) &:&\quad \sum_{i=1}^s b_i c_i^{q-1} a_{ij} &= \frac{b_j}{q} (1 - c_j^q), \quad \text{for } j = 1, \dots, s, \quad q = 1, \dots, \zeta; \\
    \end{aligned}
\end{equation}
with $p \leq \eta + \zeta + 1$ and $p \leq 2 \eta + 2$.
\reva{Recall that collocation Runge--Kutta methods satisfy $C(s)$ by definition, while discontinuous collocation Runge--Kutta methods satisfy $C(\eta)$ with $\eta < s$, see \citep[Section II.1.2]{HairerEtAl2002_geometric}.}
In terms of trees, simplifying assumption $B(p)$ is equivalent to the functional $\beta$ of the collocation \reva{or discontinuous collocation} method satisfying the order conditions on trees of height $2$ up to size $p$, that is, 
\[ \beta(\forestRB) = \frac{1}{\forestSB!}. \]
Simplifying assumption $C(\eta)$ is equivalent to the functional $\beta$ being invariant with respect to the transformation on trees that replaces a branch with $q - 1$ leaves by $q$ leaves for $1 \leq q \leq \eta$ and divides by $q$, that is, 
\[ \beta(\forestTB) = \frac{1}{q} \beta(\forestUB), \]
where $\forestVB$ is a placeholder for the rest of the tree.
Simplifying assumption $D(\zeta)$ is equivalent to the functional $\beta$ being invariant with respect to the following transformation on trees with $q-1$ leaves at the root,
\[ \beta(\forestWB) = \frac{1}{q} \beta(\forestXB - \forestYB),  \]
where $\forestAC$ is a placeholder for one or multiple branches. See \citep{HairerEtAl1993_nonstiff,HairerEtAl1996_stiff} for details.

\renewcommand{\figurename}{Table}
\begin{figure}
    \centering
    \begin{tabular}{ | c || c c c | c | }
        \hline
        Method & \multicolumn{3}{c|}{Simplifying assumptions} & order \\
        \hline
        \hline
        Gauss & $B(2s)$ & $C(s)$ & $D(s)$  & $2s$ \\
        Radau IA & $B(2s - 1)$ & $C(s - 1)$ & $D(s)$  & $2s - 1$ \\
        Radau IIA & $B(2s - 1)$ & $C(s)$ & $D(s - 1)$  & $2s - 1$ \\
        Lobatto IIIA & $B(2s - 2)$ & $C(s)$ & $D(s - 2)$  & $2s - 2$ \\
        Lobatto IIIB & $B(2s - 2)$ & $C(s - 2)$ & $D(s)$  & $2s - 2$ \\
        Lobatto IIIC & $B(2s - 2)$ & $C(s - 1)$ & $D(s - 1)$  & $2s - 2$ \\
        \hline
    \end{tabular}
    \caption{Simplifying assumptions and corresponding order of collocation \reva{and discontinuous collocation} methods with $s$ stages.}
    \label{fig:simplifying_assumptions}
\end{figure}
\renewcommand{\figurename}{Fig.}
Given an EED $\bm{A}_\Delta = (\tilde{a}_{ij})_{i,j=1}^s$, \reva{a constant $Y \in \mathbb{R}$, and a set of constants $W = (W_q)_{q=1,\dots,\eta}$, let us introduce assumptions on $\bm{A}_\Delta$ defined as} 
\begin{align*}
    C_W(\eta) &:&\quad \sum_{j=1}^s \tilde{a}_{ij} c_j^{q-1} &= c_i^q W_q, \quad \text{for } i = 1, \dots, s, \quad q = 1, \dots, \eta \,; \\
    D_Y(\zeta) &:&\quad \sum_{i=1}^s b_i c_i^{q-1} \tilde{a}_{ij} &= \sum_{j=1}^s b_j c_j^q Y, \quad \text{for } j = 1, \dots, s, \quad q = 1, \dots, \zeta \,.
\end{align*}
We note that $C(\eta)$ is a particular case of $C_W(\eta)$ \reva{with $W_q = \frac1q$}, however, $D(\zeta)$ is not a particular case of $D_Y(\zeta)$.
The main result of this section is stated in Theorem \ref{thm:order_jump} which is proved in Section \ref{sec:collocation_order_jump} \reva{and which, given an SDC method $(\bm{A}^0_\Delta, \dots, \bm{A}^{K+1}_\Delta, \bm{A}, \bm{b})$, provides conditions on $\bm{A}^k_\Delta$ for $k=1,\dots,K+1$ which guarantee an order jump.}

\begin{theorem}
    \label{thm:order_jump}
    Consider SDC with $K+1$ iterations which approximates a \reva{Runge--Kutta method of order $p$ satisfying $B(p), C(\eta), D(\zeta)$ with $p = \eta + \zeta$. Let the $K^{th}$ iteration of SDC be of iteration order $p_K$. Assume there exists a sequence of natural numbers $\eta_k \in \mathbb{N}$ and constants $W_k \in \mathbb{R}^{\eta_k}$ for $k=1,\dots, K+1$ such that $\bm{A}^k_\Delta$ satisfies $C_{\eta_k}(W_k)$ with
    \[ W_{K+1,p_K + 1} = \frac{1}{p_K + 1} \,, \quad \text{and} \quad 1 \leq \eta_k \leq \eta \,, \quad \eta_{k+1} = \eta_k + 1 \,. \]
    We have $p_{K+1} = p_K + 2$ if one of the following conditions is satisfied:
    \begin{enumerate}
        \item $p_K < \eta_{K+1}$,
        \item $\eta_{K+1} \leq p_K < \eta_{K+1} + \zeta_{K+1}$ and $\bm{A}_\Delta^k$ satisfies $D_{Y_k}(\zeta_k)$ for $k = 1, \dots, K + 1$ with
            \[ Y_{K+1} = \frac{1}{p_K + 1} \,, \quad \text{and} \quad 1 \leq \zeta_k \leq \zeta, \quad \zeta_k \leq \eta_k + 1. \]
    \end{enumerate}}
\end{theorem}

\reva{The proof of Theorem \ref{thm:order_jump} is presented in Section \ref{sec:collocation_order_jump}. Note that, in the statement of Theorem \ref{thm:order_jump}, constants $W_{k,q}$ and $Y_{k}$ for which no assumptions are placed can be arbitrary. We can then systematically build families of EEDs that benefit from repeated order jumps.}

\begin{corollary}
    \reva{A diagonal EED $\bm{A}_\Delta^k$ satisfies $C_{W_k}(\eta_k)$ and $D_{Y_k}(\zeta_k)$ for any $\eta_k$ and $\zeta_k$ if and only if $\bm{A}^k_\Delta = \diag(\bm{c})M_k$ for some constant $M_k \in \mathbb{R}$  with
        \[ M_k := Y_k = W_{k,1} = \cdots = W_{k,\eta_k} \,. \]
    In particular, there exists a unique diagonal EED $\bm{A}_\Delta^{k} = \diag(\mathbf{c})/2k$ for all $k \reva{= 1, \dots, K+1}$ such that the \reva{iteration} order of SDC increases by $2$ per iteration.}
\end{corollary}
\begin{proof}
    \revb{Let the EED $\bm{A}_\Delta^k$ be diagonal for some $k$. If $\bm{A}_\Delta^k = (\tilde{a}_{ij})_{i,j=1}^s$ satisfies $C_{W_k}(\eta_k)$ for any $\eta_k > 1$, then $\tilde{a}_{ii} = c_i W_{k,q}$ for all $q = 1, \dots, \eta_k$. Similarly, if $\bm{A}_\Delta^k$ satisfies $D_{Y_k}(\zeta_k)$ for any $\zeta_k > 1$, then, $\tilde{a}_{ii} = c_i Y_k$ with $Y_k = W_{k,1} = \cdots = W_{k,\eta_k}$. Therefore, $\bm{A}_\Delta^k = \diag(\bm{c}) M_k$ with $M_k := Y_k = W_{k,1} = \cdots = W_{k,\eta_k}$.}
    
    \revb{Note that to perform an order jump at iteration $k$ with a diagonal EED, we need to choose $M_k = \frac1{p_{k-1} + 1}$. Assuming we get an order jump at each iteration, we have $p_k = 2k + 1$, which, in the case when all EEDs are diagonal, implies $M_k = 1/2k$.}
\end{proof}

\revb{Note that the uniqueness of the diagonal EED which results in an order jump follows also from Corollary \ref{corr:unique_EED} where a linear problem is considered.}

\begin{corollary}
    \label{corr:MIN-SR-NS_jump}
    Setting $\bm{A}_{\Delta}^k= \diag(\bm{c})/s$ for all $k$, \reva{corresponding to the MIN-SR-NS EED} introduced by~\cite{CaklovicEtAl2025}, results in an SDC method that performs an order jump on $(s-1)^{th}$ iteration.
\end{corollary}
\begin{proof}
    \reva{Note that following the notation of Theorem \ref{thm:order_jump}, $\bm{A}_{\Delta}^k$ is a diagonal EED with $M_k = 1/s$ and performs an order jump at $k^{th}$ iteration if $s = p_{k-1} + 1$. Since the order of SDC is increased by $1$ per iteration with $p_0 = 1$, we have $p_{k-1} = k$, and the order jump is performed when $k = s - 1$.}
\end{proof}

\begin{remark}
Note that Theorem~\ref{thm:order_jump} applies only if $\bm{A}$ is a collocation \reva{or discontinuous collocation} method. 
However, numerical experiments not reported here suggest that many Runge–Kutta methods can serve as the underlying method. 
A possible explanation is that most high-order Runge–Kutta methods are constructed using simplifying assumptions analogous to $B(p), C(\eta), D(\zeta)$.
\end{remark}

\subsubsection{Order jumps for linear problems}
In this subsection, we focus on understanding the phenomenon of order jumps of SDC approximating a general Runge-Kutta method. We prove Proposition \ref{prop:general_order_jump} which introduces a necessary and sufficient condition \eqref{eq:general_condition} on the EED $\bm{A}_\Delta$ which results in an order jump when the problem is linear, that is, $f : \mathbb{R}^d \to \mathbb{R}^d$ of \eqref{IVP} is a degree one polynomial. 
For a general problem, the condition \eqref{eq:general_condition} is necessary, but not sufficient.

A bamboo tree is a tree which does not have any branches. 
It is the only tree whose size is equal to its height see Table \ref{table:HeightAndSize} and the only tree which appears in the B-series of Runge-Kutta methods and the exact solution when the problem is linear. 
A list of all bamboo trees up to size (and height) $5$ can be found below,
\[ \forestBC \,, \quad \forestCC \,, \quad \forestDC \,, \quad \forestEC \,, \quad \forestFC \,. \]
We use bamboo trees to derive the condition on EED $\bm{A}_\Delta$ for the order to be increased by $2$ per iteration. Let the bamboo tree of hight $h$ be denoted by
\[ \forestGC_{h}. \]
Let us define the coefficient error of the $i^{th}$ stage of SDC with $K$ iterations to be
\[ \epsilon^{[K,i]} := \beta^{[i]} - \alpha^{[K,i]}.\]
We use~\eqref{eq:ai_prop_sdc} to prove Proposition \ref{prop:general_order_jump} which derives a condition on the EED $\bm{A}_\Delta$ that results in an order jump.
Note that for linear problems, order and height order agree.

\begin{proposition}\label{prop:general_order_jump}
    Let the internal stages of SDC with $K$ iterations be of \reva{iteration} order $h_K$. An EED $\bm{A}^{K+1}_\Delta = (\tilde{a}_{ij})_{i,j=1}^s$ which results in SDC with $K+1$ iterations and internal stages of \reva{iteration} order $h_K+2$ satisfies the following condition 
    \begin{equation}
        \label{eq:general_condition}
        \sum_{j=1}^s \tilde{a}_{ij} \epsilon^{[K,j]}(\forestHC_{h_K+1}) = \sum_{j=1}^s a_{ij} \epsilon^{[K,j]}(\forestIC_{h_K+1}).
    \end{equation}
    For a linear problem, condition \eqref{eq:general_condition} is sufficient to guarantee an order jump.
\end{proposition}
\begin{proof}
    Assume iteration $K+1$ performs an order jump, then, the coefficient map of the resulting SDC method must satisfy
    \[
        \alpha^{[K+1,i]} (\forestJC_{h_K+2}) = \beta^{[i]} (\forestKC_{h_K+2}).
    \]
    We use~\eqref{eq:ai_prop_sdc} to derive the condition on $\tilde{a}_{ij}$, 
    \[ 
        \sum_{j=1}^s \big( a_{ij} - \tilde{a}_{ij} \big) \alpha^{[K,j]}(\forestLC_{h_K+1}) + \sum_{j=1}^s \tilde{a}_{ij} \alpha^{[K+1,j]} (\forestMC_{h_K+1}) = \beta^{[i]}(\forestNC_{h_K+2}),
    \]
    We recall Theorem \ref{thm:order_per_sweep} and group the $\tilde{a}_{ij}$ terms which gives,
    \[ 
        \sum_{j=1}^s a_{ij} \alpha^{[K,j]}(\forestOC_{h_K+1}) + \sum_{j=1}^s \tilde{a}_{ij} \epsilon^{[K,j]}(\forestPC_{h_K+1}) = \beta^{[i]}(\forestQC_{h_K+2}),
    \]
    We use~ \eqref{eq:ai_prop} to obtain
    \[ 
        \sum_{j=1}^s \tilde{a}_{ij} \epsilon^{[K,j]}(\forestRC_{h_K+1}) = \sum_{j=1}^s a_{ij} \epsilon^{[K,j]}(\forestSC_{h_K+1}),
    \]
    and the statement is proved.
\end{proof}

\begin{corollary}
    \label{corr:unique_EED}
     If the EED $\bm{A}_\Delta$ is diagonal, the only EED that satisfies~\eqref{eq:general_condition} is 
    \[ \tilde{a}_{ii} = \sum_{j=1}^s a_{ij} \frac{\epsilon^{[K,j]}}{\epsilon^{[K,i]}} (\forestTC_{h_K+1}) \,. \]
\end{corollary}
Note that we cannot guarantee that the unique diagonal EED which results in an order $2$ jump also produces an order $3$ jump, see Example~\ref{ex:orderJumps}.
Therefore, for general SDC methods, an order jump by two is the maximum that can be guaranteed by a diagonal EED. General EEDs have additional degrees of freedom which might permit the construction of EEDs resulting in high order jumps.

\begin{example}\label{ex:orderJumps}
    Suppose an SDC method with $K$ iterations, internal stages of \reva{iteration} order $h_K$, and diagonal EEDs $\bm{A}^k_\Delta = (\tilde{a}^k_{ij})_{i,j=1}^s$ such that $\tilde{a}_{ii}^{k} = l(k,i)c_i$ for some constants $l(k,i) \in \mathbb{R}$ and $\tilde{a}_{ij}^{k} = 0$ for $i \neq j$. We require 
    \[ \alpha^{[K+1,i]}(\forestUC_{h_K+2}) = \beta^{[i]}(\forestVC_{h_K+2}) \]
       to increase the iteration order of the internal stages by $2$.
       From \eqref{eq:ai_prop_sdc} we obtain
\begin{equation*}
    \alpha^{[k+1, i]}(\forestWC_{h_K+2}) = \sum_{j=1}^{s} a_{ij} \alpha^{[k, j]} (\forestXC_{h_K+1}) - \tilde{a}_{ii}^{k+1} \alpha^{[k, i]}(\forestYC_{h_K+1}) + \tilde{a}_{ii}^{k+1} \alpha^{[k+1, i]} (\forestAD_{h_K+1}).
\end{equation*}
Let's begin with the first iteration, given a copied initial state, i.e. $u_n^{[0,i]} = u_n$. We then require
\begin{equation*}
\alpha^{[0+1, i]}(\forestBD) = \beta^{[i]}(\forestCD).
\end{equation*}
From above we have
\begin{equation*}
\alpha^{[0+1, i]}(\forestDD) = \sum_{j=1}^s a_{ij} \alpha^{[0,j]}(\forestED) - \tilde{a}_{ii}^{1}\alpha^{[0,i]}(\forestFD) + \tilde{a}_{ii}^{1}\alpha^{[1, i]}(\forestGD).
\end{equation*}
\reva{Since $\alpha^{[0,j]}(\forestHD) = 0$ and $\alpha^{[1, i]}(\forestID) = \beta^{[i]}(\forestJD) = c_i$ we obtain $\alpha^{[0+1, i]}(\forestKD) = l(1,i) c^2_i$,
which is equivalent to
\begin{equation*}
	l(1,i) = \frac{\sum_{j=1}^s a_{ij} c_j}{c_i^2} =  \frac{1}{2} \frac{c_i^2}{c_i^2}  = \frac{1}{2} = l(1),
\end{equation*}
where the second equality uses the assumption that the Runge--Kutta method $(\bm{A}, \bm{b})$ satisfies $C(2)$.} Hence to increase the order by 2 in the first iteration we require $\tilde{a}_{ii}^{1} = c_i/2$. 
Let us check that we cannot perform an order $3$ jump,
\begin{equation*}
    \alpha^{[1, i]}(\forestLD) = l(1,i) c_i \alpha^{[1, i]}(\forestMD) = c_i^3 l(1,i)^2 = \frac{c_i^3}{4} \neq \beta^{[i]} (\forestND) \,,
\end{equation*}
if we assume the \reva{Runge--Kutta method $(\bm{A},\bm{b})$ satisfies $C(3)$.}
Next, consider an SDC method with \reva{iteration} order $1$ after the first iteration and the same assumptions as above. We require 
\begin{equation}
    \label{eq:example_require}
    \alpha^{[1+1, i]}(\forestOD) = \beta^{[i]}(\forestPD).
\end{equation}
We compute the left-hand side by using \eqref{eq:ai_prop_sdc} as follows,
\begin{align*}
\alpha^{[1+1, i]}(\forestQD) &= \sum_{j=1}^s a_{ij} \alpha^{[1,j]}(\forestRD) - \tilde{a}_{ii}^{2}\alpha^{[1,i]}(\forestSD) + \tilde{a}_{ii}^{2}\alpha^{[2, i]}(\forestTD) \\
                                    &= \reva{\sum_{j=1}^s a_{ij} l(1) c^2_j - l(2,i) l(1) c^3_i + \frac12 l(2, i) c_i^3}
\end{align*}
where we used Theorem \ref{thm:order_per_sweep} and assumption $\beta^{[i]} (\forestUD) = c_i^2/2$ \reva{implied by $C(2)$. Assuming the Runge--Kutta method $(\bm{A},\bm{b})$ satisfies $C(3)$}, identity \eqref{eq:example_require} becomes,
\[ l(1)\frac{c_i^3}{3} + l(2,i)\left(\frac{c_i^3}{2} - l(1)c_i^3 \right) = \frac{c_i^3}6 \,. \]
There are two cases to be taken care of. The first one is $l(1)=1/2$. Then the equation above holds true and $l(2,i)$ vanishes, \reva{that is, the SDC method is of iteration order $3$ due to the order jump in the first iteration regardless of the EED in the second iteration}. In the second case $l(1) \neq 1/2$. Solving for $l(2,i)$ yields
\begin{equation*}
	l(2,i) = \frac{1}{3}.
\end{equation*}
We note that the nodes $c_i$ and the constant $l(1)$ cancel, and hence, l(2) is independent of constants previously chosen. 
We finish this example by describing how to go from second to fourth order in the second iteration using a diagonal EED.
We require
\begin{equation}
	\alpha^{[2, i]}(\forestVD) = \sum_{j=1}^s a_{ij}\alpha^{[1,j]}(\forestWD) - \tilde{a}_{ii}^{2}\alpha^{[1,i]}(\forestXD) +  \tilde{a}_{ii}^{2}\alpha^{[2,i]}(\forestYD)
\label{req4}
\end{equation}
Since $\alpha^{[2,i]}(\forestAE) = c_i^3/6$ due to Theorem \ref{thm:order_per_sweep}, we only need to find
\begin{align*}
	\alpha^{[1,i]}(\forestBE) &= \sum_{j=1}^s a_{ij}\alpha^{[0,j]}(\forestCE) - \tilde{a}_{ii}^{1}\alpha^{[0,i]}(\forestDE) + \tilde{a}_{ii}^{1}\alpha^{[1,i]}(\forestEE)\\
	 &= \tilde{a}_{ii}^{1} \frac{c_i^2}{2} = \frac{c_i^3}{4},
\end{align*}
where in the last equality we used $\tilde{a}_{ii}^{1} = c_i/2$. Hence \eqref{req4} can be simplified to 
\begin{align*}
	\alpha^{[2, i]}(\forestFE) &= \sum_{j=1}^s a_{ij}\frac{c_j^3}{4} - \tilde{a}_{ii}^{2} \frac{c_i^3}{4} + \tilde{a}_{ii}^{2} \frac{c_i^3}{6} \\
				&= \frac{c_i^4}{16} + \tilde{a}_{ii}^{2}(\frac{c_i^3}{6} - \frac{c_i^3}{4})
\end{align*}
Once again letting $\tilde{a}_{ii}^{2} = l(2,i)c_i$ and \reva{requiring $\alpha^{[2, i]}(\forestGE) = c_i^4/24$}, yields
\begin{equation*}
	l(2,i) = l(2) = \frac{1}{4}.
\end{equation*}
A pattern seems to be emerging that requires $l(k)=1/(2k)$ in order to obtain the jumps in each iteration.
Continuing in this manner reveals an emerging pattern that requires $l(k)=1/(2k-v)$ in order to increase the order by 2 in the next iteration, where $v$ is number of iterations without order jumps.
\end{example}

For non-linear problems, the condition \eqref{eq:general_condition} is sufficient to perform an order jump only once. This is detailed in Proposition \ref{prop:single_order_jump}, see also Corollary \ref{corr:hk=K}.
\begin{proposition}
    \label{prop:single_order_jump}
    SDC with $K$ iterations and internal stages of height order $h_k$, performing an iteration with an EED satisfying \eqref{eq:general_condition}, results in SDC with $K+1$ iterations and internal stages of height order $h_{K+1} = h_K + 1$ and \reva{iteration} order $p_{K+1} \geq h_k + 2$.
\end{proposition}
\begin{proof}
    We use Theorem \ref{thm:order_per_sweep} and the relation between h\reva{e}ight order and \reva{iteration} order to see that $p_{K+1} \geq h_{K+1} = h_K + 1$. This implies that, $\alpha^{[K+1,i]} (\tau) = \beta^{[i]} (\tau)$ for all $ht(\tau) \leq h_K + 1$. We note that the only tree of size $h_K + 2$ whose height is greater than $h_K + 1$ is the bamboo tree. Condition \eqref{eq:general_condition} ensures that $\alpha^{[K+1,i]} (\tau) = \beta^{[i]} (\tau)$ for $\tau$ being a bamboo tree of size $h_K + 2$ which implies \reva{iteration} order $p_{K+1} \geq h_K + 2$.
\end{proof}

\subsubsection{Order jumps for non-linear problems}\label{sec:collocation_order_jump}
In this subsection, we derive the necessary and sufficient conditions for an EED $\bm{A}^k_\Delta$ to produce an order jump using the simplifying assumptions $B(p), C(\eta), D(\zeta)$ of Gauss, Lobatto, and Radau collocation \reva{and discontinuous collocation} methods described in the beginning of Section \ref{sec:order_jumps}.

\begin{lemma}\label{lemma:ai_eta_form}
    Consider an SDC with $K$ iterations \reva{approximating a Runge--Kutta method satisfying $C(\eta)$. Given $\eta_k \in \mathbb{R}$ and $W_k \in \mathbb{R}^{\eta_k}$, assume} $\bm{A}_\Delta^k$ satisfies $C_{W_k}(\eta_k)$ with
    \[ 1 \leq \eta_k \leq \eta, \quad \eta_{k+1} \leq \eta_k + 1, \quad \text{for } k = 1, \dots, K.\]
    Then, given a tree $\tau$ with $|\tau| \leq \eta_K$, we have,
    \begin{equation}\label{eq:ai_eta_form}
        \alpha^{[K,i]} (\tau) = c_i^{|\tau|} \omega_K (\tau),
    \end{equation}
    \reva{for some constant $\omega_K (\tau)$ which depends on the tree $\tau$ and $W_k$ for $k=1,\dots,K$.}
\end{lemma}
\begin{proof}
    Let us prove the statement for $K = 1$ taking into account that $\bm{A}_\Delta^0 = \bm{0}$ ($\alpha^{[0,i]} = 0$), and, therefore, given a tree $\tau = \reva{\{\tau_1 \cdots \tau_n \}}$ such that $1 < |\tau| \leq \eta_1$, we use \eqref{eq:ai_prop_sdc} to obtain
    \[ \alpha^{[1,i]} (\tau) = \sum_{j=1}^s \tilde{a}^{1}_{ij} \alpha^{[1,i]} (\tau_1) \cdots \alpha^{[1,i]} (\tau_n) = c_i^{|\tau|} \prod_{l=1}^{|\tau|} W_{1,l}, \]
    where we used induction on the size of the tree and $C_{W_1}(\eta_1)$ to prove the statement for $K=1$.
    We note that for any $K$, $\alpha^{[K,i]}(\bullet) = c_i$. Let $\tau = \reva{\{\tau_1 \cdots \tau_n \}}$ such that $|\tau| \leq \eta_K$, we use \eqref{eq:ai_prop_sdc} to write
    \begin{align*}
        \alpha^{[K,i]} (\tau) &= \sum_{j=1}^s (a_{ij} - \tilde{a}^{K}_{ij}) \reva{\alpha^{[K-1,j]}(\tau_1) \cdots \alpha^{[K-1,j]}(\tau_n)} + \sum_{j=1}^s \tilde{a}^{K}_{ij} \reva{\alpha^{[K,j]} (\tau_1) \cdots \alpha^{[K,j]} (\tau_n)} \\
        \shortintertext{by induction on the size of the tree and $K$, we get}
                              &= \sum_{j=1}^s (a_{ij} - \tilde{a}^{K}_{ij}) c_j^{|\tau|-1} \reva{\omega_{K-1}(\tau_1) \cdots \omega_{K-1}(\tau_n)} + \sum_{j=1}^s \tilde{a}^{K}_{ij} c_j^{|\tau|-1} \reva{\omega_K (\tau_1) \cdots \omega_K (\tau_n)} \\
        \shortintertext{using $C(\eta)$ and $C_{W_K}(\eta_K)$, we get}
        &= c_i^{|\tau|} \big((\frac{1}{|\tau|} - W_{K,|\tau|}) \reva{\omega_{K-1}(\tau_1) \cdots \omega_{K-1}(\tau_n)} + W_{K,|\tau|} \reva{\omega_K (\tau_1) \cdots \omega_K (\tau_n)} \big).
    \end{align*}
    We obtain the statement \eqref{eq:ai_eta_form} with
    \[ \omega_K (\tau) := (\frac{1}{|\tau|} - W_{K,|\tau|}) \reva{\omega_{K-1}(\tau_1) \cdots \omega_{K-1}(\tau_n)} + W_{K,|\tau|} \reva{\omega_K (\tau_1) \cdots \omega_K (\tau_n)},\]
    and the proof is finished.
\end{proof}
\begin{lemma}\label{lemma:ai_gen_form}
    Consider SDC with $K$ iterations which approximates a \reva{Runge--Kutta method of order $p$ satisfying $B(p), C(\eta), D(\zeta)$ with $p = \eta + \zeta$. Assume there exists a sequence of natural numbers $\eta_k \in \mathbb{N}$ and constants $W_k \in \mathbb{R}^{\eta_k}$ for $k=1,\dots, K$ such that $\bm{A}^k_\Delta$ satisfies $C_{\eta_k}(W_k)$ with
    \[ 1 \leq \eta_k \leq \eta \,, \quad \eta_{k+1} = \eta_k + 1 \,. \]
    Then, given $\tau = \{\tau_1 \cdots \tau_n\}$, we have,
    \begin{equation}\label{eq:ai_gen_form}
        \alpha^{[K]} (\tau) = \frac{1}{|\tau|} \omega_K (\tau_1) \cdots \omega_K (\tau_n),
    \end{equation}
    if one of the following conditions is satisfied:
    \begin{enumerate}
        \item $|\tau| \leq \eta_K + 1$,
        \item $\eta_K + 1 < |\tau| \leq \eta_K + \zeta_K + 1$ and $\bm{A}_\Delta^k$ satisfies $D_{Y_k}(\zeta_k)$ with $Y_k \in \mathbb{R}$ where,
            \[ 1 \leq \zeta_k \leq \zeta, \quad \zeta_k \leq \eta_k + 1 \,, \quad \text{and} \quad k = 1, \dots, K \,. \]
    \end{enumerate}}
\end{lemma}
\begin{proof}
    For a tree $\tau = \reva{\{\tau_1 \cdots \tau_n\}}$ of size $|\tau| \leq \eta_K + 1$, we use Lemma \ref{lemma:ai_eta_form} and $B(p)$ to obtain,
    \begin{align}
        \alpha^{[K]}(\tau) &= \sum_{i=1}^s b_i \reva{\alpha^{[K,i]}(\tau_1) \cdots \alpha^{[K,i]}(\tau_n)} \nonumber \\
                           &= \sum_{i=1}^s b_i c_i^{|\tau|-1} \reva{\omega_K (\tau_1) \cdots \omega_K (\tau_n)} = \frac{1}{|\tau|} \reva{\omega_K (\tau_1) \cdots \omega_K (\tau_n)},
        \label{eq:use_ai_eta_form}
    \end{align}
    which proves the statement.
    For any tree $\tau = \reva{\{\tau_1 \cdots \tau_n\}}$ of size $\eta_K + 1 < |\tau| \leq \eta_K + \zeta_K + 1$, if all $\tau_i$ have sizes less or equal to $\eta_K$, then \eqref{eq:use_ai_eta_form} proves the statement.

    Assume $\tau = \reva{\{\tau_1 \cdots \tau_n\}}$ of size $\eta_K + 1 < |\tau| \leq \eta_K + \zeta_K + 1$ and assume there exists a tree $\gamma\reva{ = \tau_j}$ \reva{for some $j$} such that $|\gamma| \geq \eta_K + 1$, then, $\reva{|\tau_1 \cdots \tau_{j-1} \tau_{j+1} \cdots \tau_n|} \leq \zeta_K - 1 \leq \eta_K$, which implies that $\gamma$ is the only tree $\tau_i$ such that $|\gamma| \geq \eta_K + 1$.
    Let $q := |\reva{\tau_1 \cdots \tau_{j-1} \tau_{j+1} \cdots \tau_n}|$. We use Lemma \ref{lemma:ai_eta_form} to obtain
    \begin{align*}
        \alpha^{[K]} (\tau) &= \sum_{i=1}^s b_i \reva{\alpha^{[K,i]} (\tau_1) \cdots \alpha^{[K,i]} (\tau_n)} \\
                            &= \sum_{i=1}^s b_i c_i^q \alpha^{[K,i]} (\gamma) \reva{\omega_K (\tau_1) \cdots \omega_K (\tau_{j-1}) \omega_K(\tau_{j+1}) \cdots \omega_K(\tau_n)} \\
                            &= \alpha^{[K]} \big( \{\bullet^q \gamma\} \big) \reva{\omega_K (\tau_1) \cdots \omega_K (\tau_{j-1}) \omega_K(\tau_{j+1}) \cdots \omega_K(\tau_n)}.
    \end{align*}
This implies that it is enough to prove the statement for trees of the form $\{\bullet^q \gamma\}$. We have $q < \zeta_K \leq \zeta$, therefore, we can apply $D_{Y_K}(\zeta_K)$ and $D(\zeta)$. Let $\gamma = \reva{\{\gamma_1 \cdots \gamma_m\}}$, then, using \eqref{eq:ai_prop_sdc},
    \begin{align*}
        \alpha^{[K]} \big( \reva{\{\bullet^q \gamma\}} \big) &= \sum_{i,j=1}^s b_i c_i^q a_{ij} \reva{ \alpha^{[K-1,j]}(\gamma_1) \cdots \alpha^{[K-1,j]}(\gamma_m)} \\
        &\quad\quad - \sum_{i,j=1}^s b_i c_i^q \tilde{a}^{K}_{ij} \reva{\alpha^{[K-1,j]}(\gamma_1) \cdots \alpha^{[K-1,j]}(\gamma_m)} \\
        &\quad\quad + \sum_{i,j=1}^s b_i c_i^q \tilde{a}^{K}_{ij} \reva{\alpha^{[K,j]}(\gamma_1) \cdots \alpha^{[K,j]}(\gamma_m)} \\
        \shortintertext{using $D(\zeta)$ and $D_{Y_K}(\zeta_K)$, we get}
        \alpha^{[K]} \big( \reva{\{\bullet^q \gamma\}} \big) &= \frac{1}{q+1} \sum_{i=1}^s b_i (1 - c_i^{q+1}) \reva{\alpha^{[K-1,i]}(\gamma_1) \cdots \alpha^{[K-1,i]}(\gamma_m)} \\
        &\quad\quad - \sum_{i=1}^s b_i c_i^{q+1} Y_K \reva{\alpha^{[K-1,i]} (\gamma_1) \cdots \alpha^{[K-1,i]} (\gamma_m)} \\
        &\quad\quad + \sum_{i=1}^s b_i c_i^{q+1} Y_K \reva{\alpha^{[K,i]}(\gamma_1) \cdots \alpha^{[K,i]}(\gamma_m)} \\
        \shortintertext{the definition of $\alpha^{[K-1]}$ and $\alpha^{[K]}$ gives}
        \alpha^{[K]} \big( \reva{\{\bullet^q \gamma\}} \big)&= \frac{1}{q+1} \big( \alpha^{[K-1]} (\gamma) - \alpha^{[K-1]} (\reva{\{\bullet^{q+1} \gamma_1 \cdots \gamma_m\}} \big) \\
                                                            &\quad \quad - Y_K \alpha^{[K-1]} \big( \reva{\{\bullet^{q+1} \gamma_1 \cdots \gamma_m\}} \big) + Y_K \alpha^{[K]} \big( \reva{\{\bullet^{q+1} \gamma_1 \cdots \gamma_m\}} \big) \\
        \shortintertext{we use induction on the size of $\reva{\gamma_1 \cdots \gamma_m}$ and the fact that $\omega_k (\bullet) = 1$ to get}
    \alpha^{[K]} \big( \reva{\{\bullet^q \gamma\}} \big) &= \big(\frac{1}{(q+1)|\gamma|} - \frac{1}{(q + 1)(|\gamma| + q + 1)} \\
                                                         &\quad\quad - \frac{1}{|\gamma| + q + 1} Y_K \big) \reva{\omega_{K-1} (\gamma_1) \cdots \omega_{K-1} (\gamma_m)} \\
                                                             &\quad\quad + \frac{1}{|\gamma| + q + 1} Y_K \reva{\omega_K (\gamma_1) \cdots \omega_K (\gamma_m)} \\
                                                             &= \frac{1}{|\gamma| + q + 1} \Big( \big( \frac{|\gamma| + q + 1}{(q+1)|\gamma|} - \frac{1}{q+1} - Y_K \big) \reva{\omega_{K-1} (\gamma_1) \cdots \omega_{K-1} (\gamma_m)} \\
        &\quad\quad + Y_K \reva{\omega_K (\gamma_1) \cdots \omega_K (\gamma_m)} \Big) \\
        \shortintertext{which reduces to}
        \alpha^{[K]} \big( \reva{\{\bullet^q \gamma\}} \big) &= \frac{1}{|\tau|} \Big( (\frac{1}{|\gamma|} - Y_K) \reva{\omega_{K-1} (\gamma_1) \cdots \omega_{K-1} (\gamma_m)} + Y_K \reva{\omega_K (\gamma_1) \cdots \omega_K (\gamma_m)} \Big) \\
                                                             &= \reva{\frac{1}{|\tau|} \omega_K(\gamma)} \,.
    \end{align*}
    where 
    \[ \omega_K (\gamma) = (\frac{1}{|\gamma|} - Y_K) \reva{\omega_{K-1} (\gamma_1) \cdots \omega_{K-1} (\gamma_m)} + Y_K \reva{\omega_K (\gamma_1) \cdots \omega_K (\gamma_m)}, \]
    which proves the statement of the lemma.
\end{proof}

We can now prove the \reva{main Theorem \ref{thm:order_jump} which characterizes EEDs that produce order jumps.}

\begin{proof}[Proof of Theorem \ref{thm:order_jump}]
    To show that $p_{K+1} = p_K + 2$ assuming $p_K + 2 \leq p$, we need to prove for all $\tau \in T$ with $|\tau| \leq p_K + 2$ that
    \begin{equation}
        \label{eq:order_jump_main}
        \reva{\alpha^{[K+1]}(\tau) = \beta(\tau) \,, \quad \text{where} \quad \beta(\tau) = 1/\tau! \,.}
    \end{equation}
    This follows for all trees $\tau \in T$ with $|\tau| \leq p_K + 1$ due to Theorem \ref{thm:order_per_sweep}. It remains to show that the same is true for all trees $\tau$ of size $|\tau| = p_K + 2$. 
    \reva{Given $\tau = \reva{\{\tau_1 \cdots \tau_n\}}$ of size $p_K + 2$, we apply Lemma \ref{lemma:ai_gen_form},}
    \[ \alpha^{[K+1]}(\tau) = \frac{1}{|\tau|} \reva{\omega_{K+1}(\tau_1) \cdots \omega_{K+1}(\tau_n)}, \]
    and it remains to show that $\omega_{K+1} (\tau_i) = 1/\tau_i!$ for $i= 1, \dots, n$.
    If $n > 1$, then all trees $\tau_i$ have size less or equal to $p_K$ and, \reva{thus, following Lemma \ref{lemma:ai_gen_form},
    \[ \omega_{K+1}(\tau_i) = (|\tau_i| + 1) \alpha^{[K+1]} (\{\tau_i\}) = \frac{|\tau_i| + 1}{\{\tau_i\}!} = \frac1{\tau_i!} \,, \]
and \eqref{eq:order_jump_main} is proved}. Assume $n = 1$, then let $\tau_1$ be denoted by $\gamma = \reva{\{ \gamma_1\cdots \gamma_m\}}$. We note that $|\gamma| = p_K + 1$. If $p_K < \eta_{K+1}$, then, from the proof of Lemma \ref{lemma:ai_eta_form}, we have
    \begin{align*}
        \omega_{K+1} (\gamma) &= (\frac{1}{|\gamma|} - W_{K,p_K + 1}) \reva{\omega_K (\gamma_1) \cdots \omega_K (\gamma_m)} \\
        &\quad\quad + W_{K,p_K + 1} \reva{\omega_{K+1} (\gamma_1) \cdots \omega_{K+1} (\gamma_m)} = \frac{1}{|\gamma|} \reva{\omega_{K+1} (\gamma_1) \cdots \omega_{K+1} (\gamma_m)},
    \end{align*}
    If $\eta_{K+1} \leq p_K < \eta_{K+1} + \zeta_{K+1}$, then, according to the proof of Lemma \ref{lemma:ai_gen_form} and depending on the structure of $\gamma$, we either apply Lemma \ref{lemma:ai_eta_form} \reva{as in the previous argument, or we obtain,}
    \begin{align*}
        \omega_{K+1} (\gamma) &= (\frac{1}{|\gamma|} - Y_{K+1}) \reva{\omega_K (\gamma_1) \cdots \omega_K (\gamma_m)} + Y_{K+1} \reva{\omega_{K+1} (\gamma_1) \cdots \omega_{K+1} (\gamma_m)} \\
    &= \frac{1}{|\gamma|} \reva{\omega_{K + 1}(\gamma_1) \cdots \omega_{K + 1}(\gamma_m)} .
    \end{align*}
    Since \reva{$|\gamma_i| \leq p_K$ for all $i=1,\dots,m$, we have $\omega_{K+1} (\gamma_i) = 1/\gamma_i!$} and we obtain
    \[ \alpha^{[K+1]} (\tau) = \frac{1}{|\tau|} \frac{1}{|\gamma|} \frac{1}{\gamma_1! \cdots \gamma_m!} = \frac{1}{\tau!}, \]
    and \eqref{eq:order_jump_main} is proved.
\end{proof}

\section{Stability and convergence order}\label{sec:numexp}

This section addresses numerically the stability properties of the new SDC methods and the convergence of the SDC iteration towards the underlying RKM as $k \to \infty$ for finite values of  $\Delta t > 0$.
We also confirm numerically the order jumps of the SDC iterations as analysed in Section \ref{sec:order_jumps} by considering the linear Dahlquist test equation
\begin{equation}
	u' = \lambda u, \qquad u(t_n) = u_n, \qquad t \in [t_n, t_{n+1}],
	\label{Dahlquist}
\end{equation}
with $\lambda \in \mathbb{C}$ and $Re(\lambda) \leq 0$ and the Euler's nonlinear rigid body dynamics~\citep{Vilmart2015}
\begin{equation}
	\label{eq:euler}
	\frac{1}{D_1 D_2 D_3} \nabla C \times \nabla H = \begin{pmatrix}
	\dot{Y}_1 \\
	\dot{Y}_2 \\
	\dot{Y}_3
	\end{pmatrix} = \begin{pmatrix}
	Y_2Y_3 \\
	Y_1Y_3 \\
	-Y_1Y_2
	\end{pmatrix}, \, \bm{Y}_0 = \begin{pmatrix}1/\sqrt{3} \\ 1 \\0 \end{pmatrix}, \, t \in [t_0 = 0, t_{end} = 10],
\end{equation}
where $D_1 = N_2 - N_3$, $D_2 = N_3 - N_1$, $D_3 = N_2 - N_1$, $C = 0.5(N_1D_1Y_1^2 + N_2D_2 Y_2^2 + N_3D_3Y_3^2)$ is the Casimir, $H = 0.5(D_1 Y_1^2 + D_2 Y_2^2 + D_3 Y_3^2)$ the Hamiltonian, $Y_i = \sqrt{N_i/D_i}\tilde{Q} A_i$ a normalised angular acceleration and $N_n$ are constants satisfying $N_1 < N_3 < N_2$ and $\tilde{Q}= \sqrt{D_1D_2D_3/(N_1N_2N_3)}$. 
We set $N_1 = 1$, $N_2 = 3$ and $N_3 = 2$. 
The initial conditions in the non-normalised amplitudes $A_i$ are given by $(1,1,0)^\intercal$.

\subsection{Numerical confirmation of convergence order}
\reva{Finally, we confirm numerically that the SDC method with order jump in every sweep reaches the theoretically expected convergence order.
Figure~\ref{fig:linConvergence} shows the absolute error at the end of the simulation in the infinity norm versus the timestep size for the linear Dahlquist equation~\eqref{Dahlquist} (left) and the nonlinear Euler rigid body problem~\eqref{eq:euler} (right).
As our theory predicts, the method gains two orders with every sweep in both the linear and nonlinear case.}
 \begin{figure}[h]
 \begin{minipage}[t]{.49\textwidth}
 \centering
\includegraphics{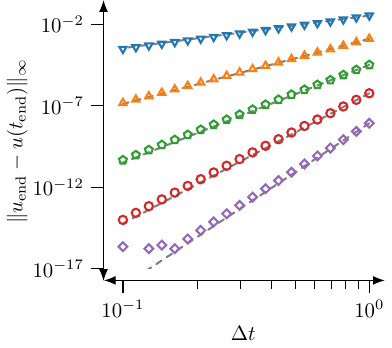}
\end{minipage} \begin{minipage}[t]{.49\textwidth}
 \centering
\includegraphics{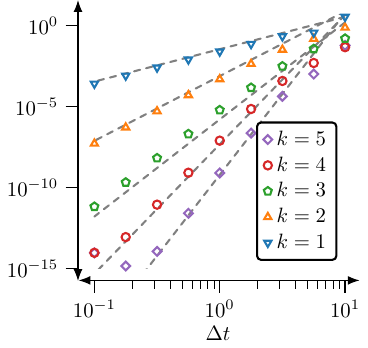}
\end{minipage}
\caption{Convergence of SDC with $M=6$ Radau IIA nodes and the EED introduced in Theorem \ref{thm:order_jump} that guarantee an order jump in every iteration. The last node was used as the solution in each timestep. The grey dashed lines serve as a guide to the eye with slopes 2, 4, 6, 8 and 10. The method shows the expected 2nd, 4th, 6th, 8th, and 10th order convergence for $k=1,2,3,4,5$ iterations. \textbf{Left}: Dahlquist test equation with $t_0=0$, $t_{\text{end}}=1$, $\lambda = -1$ and $u_0=1$.  The data point of $k=5$ close to $\Delta t = 10^{-1}$ that appears to be missing is below the shown scale. \textbf{Right}: Euler rigid body equations. Because we use SciPy \texttt{fsolve} to solve the arising nonlinear systems, the error saturates at around $\sim 10^{-14}$.}
\label{fig:linConvergence}
\end{figure}

\subsection{Stability and convergence of the SDC iteration}
To analyse convergence of SDC towards the collocation solution outside the asymptotic limit we establish some preliminaries. 
Consider an SDC method $(\bm{A}_\Delta^0, \ldots ,\bm{A}^K_\Delta, \bm{A}, \bm{b}, \bm{c})$ \reva{applied to~\eqref{Dahlquist}.}
Let $\bm{u}$ denote the stages of the collocation method~\eqref{ButcherForm} and $\bm{u}^k$ the SDC stages after $k$ iterations.
The error $\bm{e}^{k} := \bm{u}^k - \bm{u}$ is governed by 
\begin{equation}
\bm{e}^{k+1} = \bm{B}^k(z) \bm{e}^k =  \bm{B}^k \bm{e}^{k} = \bm{B}^k \dotsm \bm{B}^1 \bm{e}^1
\label{errorEvolution}
\end{equation}
 with 
 \begin{equation}
	\label{eq:error_matrix}
	\bm{B}^k(z) = z(\bm{I} - z\bm{A}{_\Delta^k})^{-1}(\bm{A}- \bm{A}_\Delta^k)
 \end{equation}
 \reva{and $z := \lambda \Delta t$}.  $\bm{B}^k(z)$ is called the iteration matrix. 
 \reva{For a constant EED, the iteration matrix remains the same for all iterations}, i.e. $\bm{B}^k(z) = \bm{B}(z)$.
 \reva{If, for a given $z$}, its spectral radius is $<1$, the method $(\bm{A}_\Delta, \bm{A}, \bm{b}, \bm{c})$ converges towards the collocation solution of $(\bm{A}, \bm{b}, \bm{c})$ as $k \to \infty$. 
 \reva{Note that while the collocation method is A-stable and thus stable for any $z \in \mathbb{C}^{-}$, the SDC iteration will not necessarily converge on the full negative complex half-plane.
 Therefore, if the spectral radius of $\bm{B}$ is smaller than unity for a given $z$, we only know that SDC will become stable if the number if iterations $K$ is large enough.
 But the precise interplay between \emph{convergence domain} and \emph{stability domain} of SDC is complex and not yet very well studied~\cite{AkramovEtAl2024}.}
 
\reva{For iterations that use a different EED in every sweep, $B^k(z)$ will not be constant and we require a different concept to assess convergence.
First, let $\Sigma = \left\{ \bm{B}^{i}(z) : i = 1, 2, \ldots \right\}$ be a sequence of iteration matrices~\eqref{eq:error_matrix} computed from a sequence of EEDs $\left\{ A^1_{\Delta}, A^2_{\Delta}, \ldots \right\}$.
Then let $ \Sigma_K := \left\{ \bm{B}^1(z), \ldots, \bm{B}^K(z) \right\} $ be the iteration matrices for the EEDs used in an SDC method with a finite number of $K$ sweeps.
Set $\rho_K (\Sigma, \Vert \cdot \Vert)\coloneq \max\{\Vert \bm{B}^{1}(z) \Vert^{\frac{1}{1}}, \ldots, \Vert \bm{B}^{K}(z) \Vert^{\frac{1}{K}} \}$} and define the joint spectral radius~\citep{rota1960} as
\begin{equation*}
	\rho(\Sigma) \coloneq \lim_{K \to \infty} \rho_K (\Sigma, \Vert \cdot \Vert).
\end{equation*}
Iteration~\eqref{errorEvolution} will converge for a selection of EEDs $\Sigma$ if $\rho(\Sigma) < 1$~\citep[Theorem 1]{Berger1992}.

\reva{We analyse convergence of the SDC method $(\bm{A}_{\Delta}^1, \dotsc, \bm{A}_{\Delta}^K, \bm{A}, \bm{b}, \bm{c})$ where $\bm{A}_{\Delta}^k = \text{diag}(\bm{c})/2k$ are the unique EEDs that generate an order jump each iteration characterized in Theorem~\ref{thm:order_jump}}.
The iteration matrix in iteration $k$ reads
\begin{equation*}
\bm{B}^k = z\left(\bm{I} - \frac{z}{2k} \text{diag}(\bm{c}) \right)^{-1}\left(\bm{A} - \frac{1}{2k} \text{diag}(\bm{c})\right)
\end{equation*}
so that 
\begin{equation*}
\lim_{k \to \infty} \bm{B}^k = z\bm{A}.
\end{equation*}
Hence, for a fixed $z$, the iteration matrix $\bm{B}^k$ approaches the iteration matrix of Picard iterations in the limit $k \to \infty$. 
Since Picard iterations define an explicit RKM, we expect the method build with EEDs that guarantee an order jump each iteration to be stable for $\vert z \vert \gg 1$ and large $k$. 
Figure~\ref{fig:stabilityInfty} shows convergence domains where the joint spectral radius $\bar{\rho}_k$ is smaller than unity and stability domains of the SDC method with order jumps and Picard iterations. 
As the number of sweeps $K$ increases, the stability domains of SDC with order jumps shrinks down to that of Picard iterations.
This is in contrast to implicit-explicit SDC for which stability regions increase with more iterations~\citep{RuprechtSpeck2016}.
 \begin{figure}[t]
 \begin{minipage}[t]{.59\textwidth}
 \centering
\includegraphics{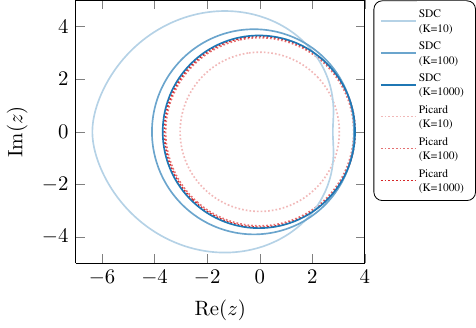}
\end{minipage} \begin{minipage}[t]{.39\textwidth}
 \centering
\includegraphics{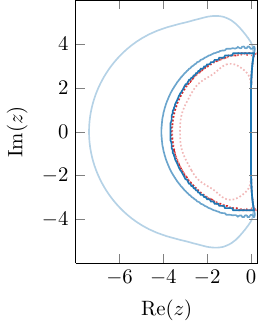}
\end{minipage}
\caption{$M=3$ Radau IIA nodes were used for the figure. \reva{\textbf{Left}: Domains where $\rho_K < 1$ for $K=10^1, 10^2, 10^3$, approximating convergence domain where the joint spectral radius is smaller than unity, for the unique EED guaranteeing order jumps in each iteration (blue solid lines) and Picard iterations (dotted red lines). \textbf{Right}: Stability domain for SDC with order jump in each iteration (blue solid lines) and Picard iteration (red dotted lines) after $k=10^1, 10^2, 10^3$ iterations.  In both plots the contours of the SDC method with order jumps approach the contours of Picard iterations as the number of sweeps is increased.}}
\label{fig:stabilityInfty}
\end{figure}

\reva{Therefore, maximizing only for order by exclusively using EEDs with order jumps will result in an SDC method suited only for non-stiff problems.  
We can, however, mix different EEDs to balance order and stability. 
Numerical experiments not documented here suggest that combinations of the MIN-SR-FLEX with order jump EEDs could deliver useful methods.
However, a detailed investigation and the development of new SDC variants that efficiently combine EEDs for stability with EEDs for order is left for future work.
}

\section{Summary}\label{sec:conclusions}

We provide a new derivation of spectral deferred corrections (SDC) which allows to formulate a wide range of different variants of SDC methods as Runge-Kutta methods (RKM) represented by Butcher tables. 
By using the theory of B-series and the methodology of simplifying assumptions for order conditions, we can provide a comprehensive analysis of SDC's convergence order.
A general proof is given that any error equation discretization (EED),  also called ``sweeper'' in SDC parlance, increases the overall order of the method by at least one, independent of the underlying nodes.
Importantly, the proof covers non-stationary EEDs where the sweeper changes in every iteration which is required for optimal convergence in some parallelizable SDC variants.
Our proof also cover EEDs that are algebraically motivated and which are therefore not consistent discretizations of the error equation.
We provide a rigorous explanation of when parallel SDC methods produce order jumps, that is, increase the order by two orders in a single iteration, a phenomenon that was observed numerically before but not yet explained theoretically.
\revb{While SDC's main area of application are stiff problems, our analysis tackles the non-stiff convergence order. We emphasize that order reduction for stiff problems has been observed for SDC~\cite{CaklovicEtAl2025} with strategies for mitigation in~\cite{Weiser2014,HuangEtAl2006}.
    The proposed examples of EEDs with repeated order jumps do not preserve the A or L-stability properties of the underliying Runge--Kutta methods, but we believe it paves the way toward new SDC methods combining both order jumps and favourable stability properties.}

\bmhead{Acknowledgements}
JF and DR thankfully acknowledge funding from the German Federal Ministry for Research, Technology and Aeronautics (BMFTR) under grant 16ME0679K. Supported by the European
Union - NextGenerationEU.
EB and GV thankfully acknowledge the support of the Swiss National Science Foundation, projects No 200020\_214819, and No. 200020\_192129 and No. 10009199.
EB thankfully acknowledges the support of the Knut and Alice Wallenberg Foundation (grant number KAW 2023.0433).

\begin{appendices}






\end{appendices}


\clearpage
\bibliography{sdc,refs}

\end{document}